\documentclass{amsart}
\usepackage{mathrsfs}
\usepackage{cancel}
\usepackage[all,hyperref]{bi-discrete}

\usepackage[backend=biber,maxbibnames=5,maxalphanames=5,style=alphabetic,bibencoding=utf8,giveninits,url=false,isbn=false]{biblatex}
\AtBeginBibliography{\small}

\usepackage{tikz}
\usepackage{tikz-3dplot} 
\usetikzlibrary{calc}

\usepackage{mathtools}

\usepackage{graphicx} 
\usepackage{subcaption}
\usepackage{pgfplots}
\usepackage{xcolor} 
\definecolor{darkgreen}{rgb}{0.0, 0.5, 0.0} 

\providecommand{\tria}{\mathcal{T}}

\providecommand{\dx}{\,\mathrm{d}x}

\providecommand{\dt}{\,\mathrm{d}t}

\providecommand{\ds}{\,\mathrm{d}s}
\providecommand{\dz}{\,\mathrm{d}z}
\providecommand{\dy}{\,\mathrm{d}y}

\providecommand{\diam}{\textup{diam}}
\providecommand{\faces}{\mathcal{F}}

\providecommand{\level}{\mathtt{lvl}}
\providecommand{\refine}{\mathtt{refine}}

\usepackage[normalem]{ulem}

\DeclareMathOperator{\Span}{span}
\DeclareMathOperator{\ran}{range}
\DeclareMathOperator{\meas}{meas}

\newcommand{\eps}{\varepsilon}

\DeclareMathOperator{\supp}{supp}

\newcommand{\nrm}{| \! | \! |}
\newcommand{\lnrm}{\mathopen{| \! | \! |}}
\newcommand{\rnrm}{\mathclose{| \! | \! |}}

\newcommand{\U}{\mathscr{U}}

\newcommand{\R}{\mathbb R}
\newcommand{\N}{\mathbb N}
\renewcommand{\P}{\mathbb P}

\newcommand{\cH}{\mathcal H}
\newcommand{\cJ}{\mathcal J}

\newcommand{\be}{\begin{equation}}
\newcommand{\ee}{\end{equation}}
\newcommand{\cL}{\mathcal L}
\newcommand{\Lis}{\cL\mathrm{is}}

\usepackage{tikz}

\pgfplotscreateplotcyclelist{MyColors}{%
    {blue,mark = *,every mark/.append style={solid,scale=0.5,fill=blue}},
    {red,mark = square*,every mark/.append style={solid,scale=0.5,fill=red}},    {dotted,blue,mark = *,every mark/.append style={solid,scale=0.5,fill=white}},
    {dotted,red,mark = square*,every mark/.append style={solid,scale=0.5,fill=white}},
{dashed,blue,mark = x,every mark/.append style={solid,scale=0.8,fill=white}},
    {dashed,red,mark =  x,every mark/.append style={solid,scale=0.8,fill=white}},    
    }

\pgfplotscreateplotcyclelist{MyColors2}{%
    {blue},
    {red},
    {darkgreen},}

\begin{document}

\author[L.\ Diening]{Lars Diening}
\author[R.\ Stevenson]{Rob Stevenson}
\author[J.\ Storn]{Johannes Storn}
\address[L.\ Diening]{Department of Mathematics, Bielefeld University, Postfach 10 01 31, 33501 Bielefeld, Germany}
\email{lars.diening@uni-bielefeld.de}
\address[R.\ Stevenson]{Korteweg-de Vries (KdV) Institute for Mathematics, University of Amsterdam, PO Box 94248, 1090 GE Amsterdam, The Netherlands}
\email{r.p.stevenson@uva.nl}
\address[J.\ Storn]{Faculty of Mathematics \& Computer Science, Institute of Mathematics, Leipzig University, Augustusplatz 10, 04109 Leipzig, Germany}
\email{johannes.storn@uni-leipzig.de}
\thanks{
The work of Lars Diening and Johannes Storn was supported by the Deutsche Forschungsgemeinschaft (DFG, German
Research Foundation) – SFB 1283/2 2021 – 317210226.}


\title[Quasi-optimal FEM for parabolic problems]{A quasi-optimal space-time FEM with local mesh refinements for parabolic problems}

\begin{abstract}

We present a space-time finite element method for the heat equation that computes quasi-optimal approximations with respect to natural norms while incorporating local mesh refinements in space-time. 
The discretized problem is solved with a conjugate gradient method with a (nearly) optimal preconditioner. 
\end{abstract}

\maketitle

\section{Introduction}
Simultaneous space-time finite element methods have gained significant attention due to their capacity for highly parallel computations and their effectiveness in handling singularities via adaptive mesh refinement in space-time. The methods have demonstrated excellent performance in a variety of numerical experiments, as highlighted in \cite{VenetieWesterdiep21,LangerSchafelner22,DanwitzVoulisHostersBehr23}. Despite these promising results, several fundamental theoretical issues remain unresolved. This paper address these challenges, advancing both the theoretical understanding and practical applicability.

In particular, we present the first finite element method for the heat equation that computes quasi-optimal approximations with respect to natural norms, while incorporating local mesh refinements in space-time. Furthermore, we introduce a preconditioner that is optimal for problems with homogeneous initial condition $u_0 =0$ or $u_0 \in H^1_0(\Omega)$ and nearly optimal for initial conditions $u_0 \in L^2(\Omega)$. This is a significant advancement, as current optimal preconditioners are restricted to tensor-product-like grids, see for example \cite{Andreev16,NeumullerSmears19,StevensonWesterdiep21,VenetieWesterdiep21}. As illustrated in numerical experiments, the resulting scheme outperforms existing space-time finite element methods such as the least-squares \cite{FuehrerKarkulik21,GantnerStevenson21,GantnerStevenson23,GantnerStevenson24} and discontinuous Petrov--Galerkin \cite{DieningStorn22} method for highly singular solutions and, additionally, can be solved efficiently by a preconditioned conjugated gradient scheme.

Our method builds upon a variational formulation employing fractional trial and test spaces, as seen in \cite{SchwabStevenson17,SteinbachZank20}. 
Since this formulation requires initial data $u_0\in H^1_0(\Omega)$, we introduce a modified formulation that allows for non-matching initial data $u_0 \in L^2(\Omega)$. 
We approximate the solution to the resulting variational problem by a minimal residual method, adhering to the abstract framework outlined in \cite{MonsuurStevensonStorn23}.
Our ansatz leads to the following three key challenges:
\begin{enumerate}[leftmargin=25pt]
\item \textbf{Quasi-optimality:} Designing a Fortin operator to ensure stability.\label{itm:a}
\item \textbf{Fractional norm evaluation:} Replacing impractical evaluations of fractional norms with an efficient preconditioner. \label{itm:b}
\item \textbf{Error control:} Deriving estimators for both adaptive mesh refinement and the iteration error.\label{itm:c}
\end{enumerate}
\noindent 
We overcome these challenges as follows. 

\textit{Ad \ref{itm:a}.}
We construct the Fortin operator by applying the standard Fortin trick, which involves defining a correction operator using bubble functions to ensure the annihilation property. However, this operator is not stable on its own, necessitating its combination with a Scott--Zhang-type interpolation operator with optimal localized approximation properties. 
According to the regularity of the underlying test space, cf.~Lemma~\ref{lem:ParaPoincareH12}, these approximation properties require finite element spaces with underlying partitions $\tria$ of the time-space cylinder $Q$ with parabolically scaled time-space cells $K = K_t \times K_x \in \tria$, that is, the length of the time interval $K_t$ scales with the squared diameter of $K_x$ in the sense that 
\begin{equation*}
|K_t| \eqsim \text{diam}(K_x)^2.
\end{equation*} 

\textit{Ad \ref{itm:b}.}
To bypass the need for directly evaluating fractional norms, we design a multi-level preconditioner. This preconditioner transforms the quadratic minimization problem into a form that can be efficiently solved using the preconditioned conjugate gradient method with optimal preconditioner for initial conditions $u_0 \in H^1_0(\Omega)$ and nearly optimal preconditioner for $u_0 \in L^2(\Omega)$. 

\textit{Ad \ref{itm:c}.}
The minimal residual method naturally generates local error contributions that drive adaptive mesh refinements. We also introduce an error estimator that monitors the distance of the current iterate from the exact discrete solution within the preconditioned conjugate gradient scheme. The combination of both allows to drive efficient adaptive schemes.\\

\textit{Outline.}
Section~\ref{sec:VarForm} introduces variational formulations of the heat equation, including a novel formulation in fractional norms for non-matching initial data $u_0\in L^2(\Omega)$. Section~\ref{subsec:discreteSetting} describes the minimal residual method at an abstract level. Section~\ref{sec:Discretization} defines the finite element spaces and the partitioning $\tria$ of the time-space domain $Q$. The construction and analysis of the Fortin and interpolation operators \ref{itm:a} are presented in Section~\ref{sec:Interpol}. Section~\ref{sec:preCond} focuses on the preconditioner design \ref{itm:b}, and Section~\ref{sec:SolvProb} discusses the optimality of the preconditioned conjugate gradient scheme and the error estimators \ref{itm:c}. 
We conclude our studies with numerical experiments in Section~\ref{sec:NumExp}.
\section{Variational formulation}\label{sec:VarForm}
In this section we derive and discuss variational formulations of the following parabolic problem.
Let $\mathcal{J} = (0,T)$ be a finite time interval, let $\Omega \subset \mathbb{R}^d$ be a bounded Lipschitz domain, and let $Q \coloneqq \mathcal{J} \times \Omega$ be the time-space cylinder.
Given initial data $u_0\colon \Omega \to \mathbb{R}$ and a source term $f\colon Q \to \mathbb{R}$, we seek $u\colon Q \to \mathbb{R}$ with
\begin{equation}\label{eq:Heat}
\begin{aligned}
\partial_t u - \Delta_x u &= f&&\quad\text{in }Q,\\
u &= 0&&\quad\text{on }\mathcal{J} \times \partial \Omega,\\
u(0,\bigcdot) &= u_0&&\quad \text{in }\Omega.
\end{aligned}
\end{equation}

More generally, we consider the following problem. 
Let $K$, $H$ be separable Hilbert spaces (for convenience over $\R$) with dense and compact embedding $K \hookrightarrow H$. Identifying $H$ with its dual, we obtain the Gelfand triple $K \hookrightarrow H \simeq H' \hookrightarrow K'$. 
Let $a(t;\bigcdot,\bigcdot)$ be a bilinear form on $K \times K$ such that $t \mapsto a(t;u,v)$ is for all $u,v\in K$ measurable. We assume that the bilinear form $a(t;\bigcdot,\bigcdot)$ is for almost all times $t \in \cJ$ bounded and coercive, that is,
\begin{equation}\label{eq:normEquia}
\begin{aligned}
|a(t;u,v)| &\lesssim \|u\|_K \|v\|_K &\quad& \text{for all }u,v\in K,\\
a(t;u,u) &\gtrsim \|u\|_K^2&& \text{for all }u\in K.
\end{aligned}
\end{equation}
We define the linear mapping $A(t) \in \cL(K,K')$ by $(A(t)\eta)(\zeta)\coloneqq a(t;\eta,\zeta)$ for $\eta,\zeta\in K$ and almost all $t\in \cJ$. Given initial data $u_0 \in H$ and $f\colon \mathcal{J} \to K'$, the resulting evolutionary problem seeks $u\colon \mathcal{J} \to K$ with
\begin{equation}\label{eq:EQabst}
\begin{aligned}
\partial_t u (t)+A(t)u(t)& = f(t) \qquad \text{for almost all }t\in \cJ,\\
u(0) & = u_0.
\end{aligned}
\end{equation}
We use $\langle\bigcdot,\bigcdot\rangle$ both to denote the scalar product on $H \times H$ and its extension to a duality pairing on $K' \times K$. Moreover, we set for trial and test functions $v$ and $w$ the operator
\begin{equation}\label{eq:defB}
(Bv)(w) \coloneqq \int_\cJ \langle \partial_t v(t),w(t)\rangle+a(t;v(t),w(t))\dt.
\end{equation} 
With the first equation in \eqref{eq:EQabst} in variational form, we seek $u \colon  \mathcal{J} \to K$ with $u(0) = u_0 \in H$ and
\begin{equation*}
b(u,v) \coloneqq (Bu)(v) = \int_\mathcal{J} f(t) v(t)\dt.
\end{equation*}
In other words, with $\gamma_0 u\coloneqq u(0)$ we seek  the solution to the operator equation
\begin{equation}\label{eq:defBe}
B_e u\coloneqq (Bu,\gamma_0 u)=(f,u_0).
\end{equation}
In the remainder of this section we introduce and discuss suitable interpretations of the operator equation above.
\begin{remark}[G\aa rding inequality]
As for example illustrated in \cite[Rem.~65.5]{ErnGuermond21c}, it is possible to replace the second assumption in \eqref{eq:normEquia} by the G\aa rding inequality, reading with some arbitrary constant $\lambda \geq 0$
\begin{equation*}
\|u\|_K^2 \lesssim a(t;u,u) + \lambda\, \lVert u \rVert_H^2\qquad \text{for all }u\in K.
\end{equation*}
\end{remark}
\subsection{Classical setting}\label{subsec:ClassicalSetting}
The classical setting involves the Bochner spaces
\begin{align*}
U_C & \coloneqq L^2(\mathcal{J};K) \cap H^1(\mathcal{J},K'),\\
V_C & \coloneqq L^2(\mathcal{J};K).
\end{align*}
We consider the operator $B$ as a mapping from $U_C$ into the dual $V_C' \eqsim L^2(\cJ;K')$ of $V_C$. 
The following existence result is proven in \cite[Thm.~6.6]{ErnGuermond04} and \cite[Thm.~5.1]{SchwabStevenson09}, see also \cite[Chap.~IV, Sec.~26]{Wloka82} and \cite[Chap.~XVIII, Sec.~3]{DautrayLions92}.
\begin{theorem}[Well-posedness of classical setting]\label{thm:classicalEx}
The operator $B_e$ in \eqref{eq:defBe} is a linear isomorphism from $U_C$ into the dual space $(V_C \times H)' \eqsim V_C'\times H$, that is, 
\begin{equation*}
B_e \in \Lis\big(U_C,V_C'\times H\big).
\end{equation*} 
The upper bound for the norm of the operator and that of its inverse only depend on the hidden constants in \eqref{eq:normEquia}, and on $T$ when it tends to zero.
\end{theorem}
To date, to the best of the authors' knowledge, no numerical scheme has been developed that enables quasi-optimal finite element schemes with non-tensor meshes in the context of the functional setting outlined in Theorem~\ref{thm:classicalEx}. The difficulties associated with finite element approximations in this functional setting have been elucidated in \cite{StevensonStorn22}. Consequently, we employ the subsequent formulations.
\subsection{Fractional setting for homogeneous initial data}\label{subsec:FracSetting}
The lack of regularity with respect to time in $V_C$ causes difficulties in the design of Fortin operators and thus of inf-sup stable numerical schemes on non-tensor meshes.
Similarly observations apply to ultra-weak formulations like in \cite[Sec.~4]{TantardiniVeeser16} due to the trial space's lack of regularity with respect to time.
Further challenges result from the treatment of the dual norm $H^1(\cJ,K')$ on non-tensor meshes.
Consequently, we propose similar to \cite{SchwabStevenson17} and \cite{SteinbachZank20} the following ``almost symmetric'' framework that balances the regularity in the trial and test space and avoids dual norms. 

We define the Bochner spaces $H^1_{0,}(\cJ,H) \coloneqq \lbrace v \in H^1(\cJ,H) \colon v(0) = 0\rbrace$ and $H^1_{,0}(\cJ,H) \coloneqq \lbrace v \in H^1(\cJ,H) \colon v(T) = 0\rbrace$. Moreover, we set via real interpolation \cite{LionsPeetre64} the fractional Sobolev spaces
\begin{equation}\label{eq:InterpolSpaceTime}
\begin{aligned}
H^{1/2}(\mathcal{J};H) & \coloneqq [H^1(\mathcal{J};H),L^2(\mathcal{J};H)]_{1/2,2},\\
H^{1/2}_{0,}(\mathcal{J};H) & \coloneqq [H_{0,}^1(\mathcal{J};H),L^2(\mathcal{J};H)]_{1/2,2},\\
H^{1/2}_{,0}(\mathcal{J};H) & \coloneqq [H^1_{,0}(\mathcal{J};H),L^2(\mathcal{J};H)]_{1/2,2}.
\end{aligned}
\end{equation}
Set for all $v\in L^2(\cJ;H)$ the squared semi-norm 
\begin{equation*}
\lvert v \rvert_{H^{1/2}(\mathcal{J};H)}^2 \coloneqq \int_\mathcal{J} \int_\mathcal{J} \frac{\lVert v(t) - v(s)\rVert_H^2}{|t-s|^2} \dt\ds.
\end{equation*}
\begin{proposition}[Equivalent characterizations]\label{prop:equiNorms}
Let $v\in L^2(\cJ;H)$.
The norm in the first interpolation space in \eqref{eq:InterpolSpaceTime} is equivalent to the Sobolev-Slobodeckij norm 
\begin{align*}
\lVert v \rVert_{H^{1/2}(\mathcal{J};H)} \coloneqq \big(\lVert v \rVert_{L^2(\mathcal{J};H)}^2 + \lvert v \rvert_{H^{1/2}(\mathcal{J};H)}^2\big)^{1/2}.
\end{align*}
The norms in the second and third space in \eqref{eq:InterpolSpaceTime} are equivalent to the norms 
\begin{align*}
\lVert v \rVert_{H_{0,}^{1/2}(\mathcal{J};H)} & \coloneqq \Big(\lVert v \rVert_{H^{1/2}(\mathcal{J};H)}^2  + \int_{\cJ} \frac{\lVert v(t) \rVert_H^2}{t} \dt \Big)^{1/2}&&\text{and}\\
\lVert v \rVert_{H_{,0}^{1/2}(\mathcal{J};H)} & \coloneqq \Big(\lVert v \rVert_{H^{1/2}(\mathcal{J};H)}^2  + \int_{\cJ} \frac{\lVert v(t) \rVert_H^2}{T-t} \dt \Big)^{1/2},&&\text{respectively}.
\end{align*}
\end{proposition}
\begin{proof}
This classical result can be found for example in \cite{LionsMagenes72}.
\end{proof}
\begin{remark}[Extension by zero]\label{rem:ExtByZero}
One can characterize the latter spaces in \eqref{eq:InterpolSpaceTime} via extensions by zero in the sense that any $v\in H^{1/2}_{0,}(\mathcal{J};H)$ can be seen as a function in $H^{1/2}(-\infty,T;L^2(\Omega))$ with $v|_{(-\infty,T]\times \Omega} = 0$, since such functions satisfy
\begin{equation}\label{eq:HalfSpaceToLocal}
\begin{aligned}
&\lVert v \rVert_{H^{1/2}(-\infty,T;H)}^2 = \lVert v \rVert_{H^{1/2}(\mathcal{J};H)}^2 + 2 \int_0^T \int_{-\infty}^0 \frac{\lVert v(t)\rVert_H^2}{|t-s|^2} \ds \dt \\
&\qquad = \lVert v \rVert_{H^{1/2}(\mathcal{J};H)}^2 + 2 \int_0^T \frac{\lVert v(t)\rVert_H^2}{t} \dt 
\eqsim \lVert v \rVert_{H^{1/2}_{0,}(\mathcal{J};H)}^2.
\end{aligned}
\end{equation}
A corresponding result holds for $H^{1/2}_{,0}(\mathcal{J};H)$.
\end{remark}
We set the fractional trial and test spaces
\begin{equation}\label{eq:DefUvV}
\begin{aligned}
U_F& \coloneqq L^2(\mathcal{J};K) \cap H^{1/2}_{0,}(\mathcal{J};H),\\ 
V_F& \coloneqq L^2(\mathcal{J};K) \cap H^{1/2}_{,0}(\mathcal{J};H).
\end{aligned}
\end{equation}
\begin{theorem}[Well-posedness of fractional setting]\label{thm:fractEx}
The operator $B$ in \eqref{eq:defB} is a linear isomorphism from $U_F$ into the dual space $V_F'$, that is, 
\begin{equation*}
B \in \Lis(U_F,V_F').
\end{equation*} 
The upper bound for the norm of the operator and that of its inverse only depend on the hidden constants in \eqref{eq:normEquia}. 
\end{theorem}
\begin{proof}
For $T=\infty$, a proof of this result can be found in \cite[Thm.~2.2]{BaiocchiBrezzi83} and in \cite[Thm.~4.3]{Fontes09}.
Proofs for $T<\infty$ can be found in \cite[Cor.~3.9]{SchwabStevenson17} using interpolation spaces, and in \cite[Thm.~3.2]{SteinbachZank20} using a Hilbert transformation.
\end{proof}
In contrast to the formulation presented in Theorem~\ref{thm:classicalEx}, the formulation of this theorem features trial and test spaces with balanced regularity and without dual norms. This aspect is advantageous for the design of finite element methods that utilize non-tensor spaces. 
It requires, however, homogeneous initial data.
\subsection{Fractional setting with inhomogeneous initial data}\label{subsec:MixedProb}
As we have seen, finding $u \in U_F$ such that $B u=f$ with $f \in V_F'$ is a well-posed variational formulation of the parabolic problem \eqref{eq:EQabst} with homogeneous initial data $u_0=0$.
When having non-zero initial data $u_0$, an option is to extend it to a function $\bar{u}\colon \cJ \rightarrow K$, and to compute $u-\bar{u} \in U_F$ with $B (u-\bar{u})=f - B \bar{u}$, assuming that $B\bar{u} \in V_F'$. When $u_0 \in K$, it can be trivially extended to $\bar{u} \in H^1(\cJ;K)$, so that $B\bar{u} \in V_F'$.
For $u_0 \not\in K$, the task of constructing such a function $\bar{u}$ is non-trivial. We circumvent this challenge by introducing the following variational formulation that performs this task on the fly.
This formulation involves the spaces 
\begin{align*}
U & \coloneqq U_F + U_C = \big(L^2(\cJ;K) \cap H_{0,}^{1/2}(\cJ;H)\big) + \big(L^2(\cJ;K) \cap H^{1}(\cJ;K')\big),\\
V & \coloneqq V_F = L^2(\cJ;K) \cap H^{1/2}_{,0}(\cJ;H).
\end{align*}
The squared norm in $U$ reads 
\begin{equation*}
\lVert u \rVert_U^2 \coloneqq \inf_{u = u_1 + u_2} \big(\lVert u_1 \rVert_{U_F}^2 + \lVert u_2 \rVert_{U_C}^2\big) \qquad\text{for all }u\in U.
\end{equation*}
Additionally, we set the space $H^1_{0,}(\cJ;K') \coloneqq \lbrace v \in H^1(\cJ;K') \colon v(0) = 0\rbrace$. 

\begin{lemma}[Embedding]\label{lem:Embedding}
One has 
\begin{enumerate}
\item the embedding $L^2(\cJ;K) \cap H^1_{0,}(\cJ;K') \hookrightarrow U_F$, and \label{itm:embed}
\item the identity $\lbrace u \in U\colon u(0) = 0\rbrace = U_F$ with equivalent norms. \label{itm:ident}
\end{enumerate}
\end{lemma}
\begin{proof}[Proof of \ref{itm:embed}]
Define for all $u,v\in K$ the operator $(A_x u)(v)\coloneqq \langle u,v\rangle_K$. We restrict its domain to 
\begin{equation*}
D(A_x)\coloneqq \Big\{u \in K\colon \sup_{0 \neq v \in K} \frac{(A_x u)(v)}{\|v\|_H}<\infty\Big\}.
\end{equation*}
This operator $A_x$ is a self-adjoint, densely defined, unbounded operator on $H$.
The set of normalized eigenvectors $(\phi_n)_{n \in \mathbb{N}} \subset K$ of $A_x$, with corresponding eigenvalues $(\lambda_n)_{n\in \mathbb{N}} \subset \mathbb{R}_{>0}$, forms an orthonormal basis for $H$ and 
$(\lambda_n^{-1} \phi_n)_{n\in \mathbb{N}}$ forms an orthonormal basis for $K$. 
Similarly, we set $(A_t u)(v)\coloneqq \langle u,v\rangle_{H^1(\cJ)}$ for all $u,v \in H_{0,}^1(\cJ)$ with restricted domain 
\begin{equation*}
D(A_t)\coloneqq \Big\{u \in H_{0,}^1(\cJ)\colon \sup_{0 \neq v \in H_{0,}^1(\cJ)} \frac{(A_t u)(v)}{\|v\|_{L^2(\cJ)}}<\infty\Big\}.
\end{equation*}
The set of normalized eigenvectors $(\psi_n)_{n\in \mathbb{N}} \subset L^2(\cJ)$ of $A_t$, with corresponding eigenvalues $(\mu_n)_{n\in \mathbb{N}} \subset \mathbb{R}_{>0}$, forms an orthonormal basis for $L^2(\cJ)$ and $(\mu_n^{-1} \psi_n)_{n\in \mathbb{N}}$ forms an orthonormal basis for $H_{0,}^1(\cJ)$. 
By $\frac12(\eta^{1/2}+\zeta^{1/2}) \leq (\eta+\zeta)^{1/2}\leq \eta^{1/2}+\zeta^{1/2}$ for non-negative numbers $\eta,\zeta\geq 0$,
we obtain for $u=\sum_{n,m \in \mathbb{N}} u_{n,m} \psi_n \otimes \phi_m$ the equivalence of norms
\begin{align*}
&\|u\|^2_{L^2(\cJ;K) \cap H_{0,}^{1/2}(\cJ;H)} =
\sum_{n,n} (\lambda_m +\mu_n^{1/2})\, |u_{n,m}|^2 \eqsim \sum_{n,n}  \left(\lambda_m +\frac{\mu_n}{\lambda_m}\right)^{1/2} \lambda_m^{1/2} |u_{n,m}|^2 \\
 &\qquad =
\|u\|^2_{[L^2(\cJ;K)\cap H_{0,}^1(\cJ;K'),L^2(\cJ;K)]_{1/2,2}}.
\end{align*}
In particular, one has 
\begin{equation*}
U_F \simeq [L^2(\cJ;K)\cap H_{0,}^1(\cJ;K'),L^2(\cJ;K)]_{1/2,2}  \hookleftarrow L^2(\cJ;K) \cap H^1_{0,}(\cJ;K').
\end{equation*} 

\textit{Proof of \ref{itm:ident}}.
Let $u\in U = U_C + U_F$ with $u(0) = 0$. Any decomposition $u = u_1 + u_2$ with $u_1 \in U_C$ and $u_2 \in U_F$ must satisfy $u_1(0) = u(0) = 0$, that is, $u_1 \in L^2(\cJ,K) \cap H^1_{0,}(\cJ,K')$. Hence, the embedding in \ref{itm:embed} shows 
\begin{align*}
\lVert u \rVert_U^2 = \inf_{u = u_1 + u_2} \big( \lVert u_1 \rVert_{U_C}^2 + \lVert u_2 \rVert_{U_F}^2 \big) \gtrsim \inf_{u = u_1 + u_2} \big( \lVert u_1 \rVert_{U_F}^2 + \lVert u_2 \rVert_{U_F}^2 \big) \geq \tfrac12 \lVert u \rVert_{U_F}^2. 
\end{align*} 
Equivalence of norms follows by using the split
\begin{equation*}
\lVert u \rVert_U^2 = \inf_{u = u_1 + u_2} \big(\lVert u_1 \rVert_{U_C}^2 + \lVert u_2 \rVert_{U_F}^2\big) \leq \lVert 0 \rVert_{U_C}^2 + \lVert u \rVert_{U_F}^2.\qedhere
\end{equation*}
\end{proof}
With this auxiliary result we can show the following theorem.
\begin{theorem}[Well-posedness with general initial data]\label{thm:well-posedness}
The operator $B_e$ in \eqref{eq:defBe} is a linear isomorphism from $U  $ into the dual space $(V \times H)' \eqsim V' \times H$, that is, 
\begin{equation*}
B_e \in \Lis(U ,V ' \times H).
\end{equation*}
The upper bound for the norm of the operator and that of its inverse only depend linearly on the hidden constants in \eqref{eq:normEquia} and on $T$ when it tends to zero.
\end{theorem}
\begin{proof} 
First we show the boundedness of $B_e$.
Let $u=u_C+u_F$ with $u_C \in U_C$ and $u_F \in U_F$. This yields $B_e u=(B u_C+B u_F, u_C(0))$. From $V=V_F \hookrightarrow V_C$, and so 
$V_C' \hookrightarrow V_F'$, we obtain with Theorem~\ref{thm:classicalEx} and ~\ref{thm:fractEx}
\begin{equation*}
\|B u_C+B u_F\|_{V'}^2+\|u_C(0)\|_H^2 \lesssim \|B_e u_C\|_{V_C'\times H}^2+\|B u_F\|_{V'_F}^2 \lesssim \|u_C\|_{U_C}^2+\|u_F\|_{U_C}^2.
\end{equation*}
Since this holds true for any splitting $u=u_C+u_F$, we conclude $B_e \in \cL(U ,V ' \times H)$.

Now let $(f,u_0) \in V_F' \times H$ be some arbitrary data. According to Theorems~\ref{thm:classicalEx} and ~\ref{thm:fractEx} there exist unique functions $u_C \in U_C$ and $u_F \in U_F$ with 
\begin{align*}
\begin{aligned}
B_e u_C &= (0,u_0) &&\quad \text{and}&&\quad \lVert u_C \rVert_{U_C} \eqsim \lVert u_0 \rVert_{H},\\
B u_F &= f && \quad\text{and}&&\quad \lVert u_F \rVert_{U_F} \eqsim \lVert f\rVert_{V_F'}.
\end{aligned}
\end{align*} 
Hence, $u \coloneqq u_C + u_F \in U = U_C + U_F$ solves $B_e u = (f,u_0)$ and the solution depends continuously on the data, that is
\begin{equation*}
\lVert u \rVert_U^2 \leq \lVert u_C \rVert^2_{U_C} + \lVert u_F \rVert^2_{U_F} \lesssim \lVert u_0 \rVert^2_H + \lVert f\rVert_{V_F'}^2.
\end{equation*}

It remains to show injectivity of the operator $B_e$. Let $u_1,u_2\in U$ satisfy $B_e (u_1-u_2) = 0$. Then its initial trace $(u_1-u_2)(0) = 0$ equals zero and hence Lemma~\ref{lem:Embedding}\ref{itm:ident} yields $u_1-u_2 \in U_F$. The operator $B\colon U_F \to V_F'$ is a linear isomorphism (Theorem~\ref{thm:fractEx}), thus the property $B(u_1-u_2) = 0$ implies $u_1 = u_2$. 
\end{proof}
\begin{remark}[Formulation in the half-space]
Our variational formulation in Theorem~\ref{thm:well-posedness} is motivated by similar results on the half-space $\cJ = (0,\infty)$ in \cite[Thm.~2]{Tomarelli83} and \cite[Thm.~4.5]{Fontes09}.
\end{remark}

We have established a well-defined variational formulation that accommodates general initial data $u_0 \in H$. However, this achievement comes at the cost of utilizing a non-standard trial space $U = U_F + U_C$. 
We thus conclude this subsection with the subsequent discussion, showing that $U$ is only ``slightly'' smaller than $L^2(\cJ;K) \cap H^{1/2}(\cJ;H)$.
The result uses the real interpolation spaces, with $p>2$, 
\begin{equation*}
\begin{aligned}
B^{1/2,p,2}(\cJ) &\coloneqq [W^{1,p}(\cJ),L^p(\cJ)]_{1/2,2},\\
B^{1/2,p,2}(\cJ;H) &\coloneqq [W^{1,p}(\cJ;H),L^p(\cJ;H)]_{1/2,2}.
\end{aligned}
\end{equation*}
\begin{proposition}[Embedding of $U$] \label{prop:2}
One has for all $p>2$ the embeddings
\begin{equation*}
L^2(\cJ;K) \cap \big(H_{0,}^{1/2}(\cJ;H)+ B^{1/2,p,2}(\cJ;H)\big) \hookrightarrow U \hookrightarrow L^2(\cJ;K) \cap H^{1/2}(\cJ;H).
\end{equation*}
Both inclusions are strict. 
\end{proposition}

\begin{proof} 
The right inclusion follows from $U_F \hookrightarrow L^2(\cJ;K) \cap H^{1/2}(\cJ;H)$ and $U_C \hookrightarrow L^2(\cJ;K) \cap H^{1/2}(\cJ;H)$, the latter using analogous arguments as in the proof of Lemma~\ref{lem:Embedding}\ref{itm:embed}.
We prove the left inclusion by showing that
\be \label{eq:5}
B_e \in \cL(L^2(\cJ;K) \cap (H_{0,}^{1/2}(\cJ;H)+ B^{1/2,p,2}(\cJ;H)),V' \times H).
\ee
This implies by Theorem~\ref{thm:well-posedness} that for any $u \in L^2(\cJ;K) \cap (H_{0,}^{1/2}(\cJ;H)+ B^{1/2,p,2}(\cJ;H))$
\begin{equation*}
\|u\|_{L^2(\cJ;K) \cap (H_{0,}^{1/2}(\cJ;H)+ B^{1/2,p,2}(\cJ;H))} \gtrsim \|B_e u\|_{V' \times H} \eqsim \|u\|_U.
\end{equation*}

For $u \in L^2(\cJ;K) \cap (H_{0,}^{1/2}(\cJ;H)+ B^{1/2,p,2}(\cJ;H))$, let $u=u_1+u_2$ with $u_1 \in H_{0,}^{1/2}(\cJ;H)$ and $u_2 \in B^{1/2,p,2}(\cJ;H)$. We need to show that
\begin{equation*}
\|B_e u\|_{V' \times H} \lesssim \|u\|_{L^2(\cJ;K)}+\|u_1\|_{H_{0,}^{1/2}(\cJ;H)}+\|u_2\|_{B^{1/2,p,2}(\cJ;H)}.
\end{equation*}
Sufficient are the following estimates for all $v\in V$:
\begin{align} \label{eq:a}
\|u(0)\|_H  &\lesssim \|u_2\|_{B^{1/2,p,2}(\cJ;H)},\\ \label{eq:b}
\int_{\cJ} \langle \partial_t u_1(t),v(t)\rangle \dt &\lesssim \|u_1\|_{H_{0,}^{1/2}(\cJ;H)}\|v\|_{H_{,0}^{1/2}(\cJ;H)},\\ \label{eq:c}
\int_{\cJ} \langle \partial_t u_2(t),v(t)\rangle\dt &\lesssim \|u_2\|_{B^{1/2,p,2}(\cJ;H)}\|v\|_{H_{,0}^{1/2}(\cJ;H)},\\ \label{eq:d}
\int_\cJ a(t;u(t),v(t))\dt &\lesssim \|u\|_{L_2(\cJ;K)} \|v\|_{L_2(\cJ;K)}.
\end{align}
Estimate \eqref{eq:d} follows immediately by the assumption in \eqref{eq:normEquia}. 
From $B^{1/2,p,2}(\cJ) \hookrightarrow C(\cJ)$ by $p>2$ \cite[Thm.~7.34(c)]{MR2424078}, we have
\begin{equation*}
\|u(0)\|_H=\|u_2(0)\|_H\lesssim \|u_2\|_{B^{1/2,p,2}(\cJ;H)}.
\end{equation*}
Hence, \eqref{eq:a} is valid.
H\"older's inequality reveals with $1/p+1/q=1$ for smooth functions $w\colon \cJ \rightarrow \R$ and $z \in C_0^\infty(\cJ)$ that
\begin{align*}
\int_{\cJ} \partial_t w \, z\dt &\leq \|w\|_{W^{1,p}(\cJ)} \|z\|_{L^q(\cJ)},\\
\int_{\cJ} \partial_t w\, z\dt &= \int_{\cJ} w\, \partial_t z\dt
\leq  \|w\|_{L^p(\cJ)} \|z\|_{W^{1,q}(\cJ)}.
\end{align*}
Thus, a density argument and interpolation shows for all $s \in [0,1]$
\begin{equation*}
\|\partial_t \|_{\cL(
[W^{1,p}(\cJ),L^p(\cJ)]_{s,2},
[W_0^{1,q}(\cJ),L^q(\cJ)]_{1-s,2}'
)} \leq 1.
\end{equation*}
Thanks to $q<2$ it holds that
\begin{equation*}
[W_0^{1,q}(\cJ),L^q(\cJ)]_{1/2,2} \simeq [W^{1,q}(\cJ),L^q(\cJ)]_{1/2,2} \hookleftarrow H_{,0}^{1/2}(\cJ).
\end{equation*}
Combining these results and using the definition of the space $B^{1/2,p,2}(\cJ;H) = [W^{1,p}(\cJ;H),L^p(\cJ;H)]_{1/2,2}$ lead to
\begin{equation*}
\begin{aligned}
\int_{\cJ} \langle \partial_t u_2(t),v(t)\rangle \dt 
& \leq \|u_2\|_{B^{1/2,p,2}(\cJ;H)}\|v\|_{[W_0^{1,q}(\cJ;H),L^q(\cJ;H)]_{1/2,2}} \\
& \lesssim \|u_2\|_{B^{1/2,p,2}(\cJ;H)}\|v\|_{H_{,0}^{1/2}(\cJ;H)}.
\end{aligned}
\end{equation*}
Hence, \eqref{eq:c} is valid.
Taking smooth functions $w,z\colon \cJ \rightarrow \R$, where now $w$ vanishes at $0$ and $z$ vanishes at $T$, analogous arguments show
\begin{equation*}
\|\partial_t \|_{\cL([H^1_{0,}(\cJ),L^2(\Omega)]_{s,2},[H^1_{,0}(\cJ),L^2(\cJ)]_{1-s,2}')} \leq 1\qquad\text{for all }s\in [0,1].
\end{equation*}
This yields for $s=1/2$ the bound in \eqref{eq:b}.
\end{proof}
\begin{remark}[Embedding for $s > 1/2$] \label{rem:remmie}
Let $s > 1/2$ and $p>2$. Combining the embedding $H^s(\cJ;H)\hookrightarrow B^{1/2,p,2}(\cJ;H)$ and Proposition~\ref{prop:2} leads to the statement
\begin{equation*}
L^2(\cJ;K) \cap \big(H_{0,}^{1/2}(\cJ;H)+ H^{s}(\cJ;H)\big) \hookrightarrow U \hookrightarrow L^2(\cJ;K) \cap H^{1/2}(\cJ;H).
\end{equation*}
\end{remark}

\section{Minimal residual method: Overview}\label{subsec:discreteSetting}
We use a minimal residual method to approximate the solution to \eqref{eq:defBe} with given right-hand side $f\in V'$ and initial data $u_0 \in H$. 
More precisely, given discretizations $U_h \subset U$ and $V_h \subset V$, we compute the minimizer
\begin{equation}\label{eq:uhprime}
u_h = \argmin_{w_h \in U_h}\, \lVert B w_h - f\rVert_{V_h'}^2  + \lVert w_h(0) - u_0\rVert_{H}^2.
\end{equation}
If there exists a Fortin operator, that is, if there is a uniformly bounded operator $F\colon V \to V_h$ with $b(u_h,v-Fv) = 0$ for all $u_h \in U_h$ and $v \in V$, then the numerical scheme in \eqref{eq:uhprime} is quasi-optimal \cite[Lem.~3.2 as well as Thm.~3.3 and 3.6]{MonsuurStevensonStorn23}. That is, with solution $u\in U$ to $B_e u = (f,u_0)$ we obtain the quasi-best approximation property
\begin{equation}\label{eq:QuasiBest}
\lVert u-u_h \rVert_U \lesssim \min_{w_h \in U_h} \lVert u- w_h \rVert_U.
\end{equation}
We verify the existence of a Fortin operator in Theorem~\ref{thm:FortinOperator} below.

The minimization in \eqref{eq:uhprime} is equivalent \cite[Sec.~3.2]{MonsuurStevensonStorn23} to the following mixed system involving the inner product $\langle \bigcdot,\bigcdot \rangle_V$ in $V$: Seek $(\lambda_h,u_h) \in V_h \times U_h$ with
\begin{equation}\label{eq:SaddlePointProb}
\begin{aligned}
\langle \lambda_h,\mu_h\rangle_V +b(u_h,\mu_h)  &= f(\mu_h) &&\text{ for all } \mu_h\in V_h,\\
b(w_h,\lambda_h) - \langle u_h(0),w_h(0)\rangle_{H} & =-\langle u_0,w_h(0)\rangle_H&&\text{ for all }w_h\in U_h.
\end{aligned}
\end{equation}
However, it is unclear how to evaluate the inner product in the fractional space $V$ on non-product type meshes, cf.~\cite{SteinbachZank20,Zank20} for the discussion of the evaluation of the norm in $V$ on tensor product meshes.
To circumvent the evaluation of $\langle \bigcdot,\bigcdot\rangle_V$ on $V_h$, let $G_h \in \cL(V_h',V_h)$
be such that the application $G_h$ can be performed in linear time, and its inverse defines a scalar product $(G_h^{-1} \bigcdot)(\bigcdot)$ on $V_h \times V_h$ whose induced norm is uniformly equivalent to the norm $\|\bigcdot\|_{V}$ on $V_h$.
Such an operator $G_h$ is known as a preconditioner for the Riesz map $v_h\mapsto \langle v_h,\bigcdot \rangle_V \in \Lis(V_h,V_h')$, and is designed in Section~\ref{sec:preCond} below.
We replace the inner product  $\langle \bigcdot,\bigcdot\rangle_V$ in the saddle point problem \eqref{eq:SaddlePointProb} by $(G_h^{-1} \bigcdot)(\bigcdot)$.
The resulting solution $u_h \in U_h$ is still a quasi-best approximation in the sense of \eqref{eq:QuasiBest}, see \cite[Thm.~3.5]{MonsuurStevensonStorn23}. 
Elimination of $\lambda_h$ in \eqref{eq:SaddlePointProb} leads to the symmetric positive definite Schur complement equation: Seek $u_h \in U_h$ with 
\begin{equation}\label{eq:system}
(B w_h) (G_h (Bu_h-f))+\langle w_h(0),u_h(0)-u_0\rangle_{H}=0\qquad \text{for all }w_h \in U_h.
\end{equation}
The matrix corresponding to \eqref{eq:system} will be densely populated, so applying a direct solver is not feasible. Instead, we apply the preconditioned conjugate gradient method discussed in Section~\ref{subsec:PCGscheme} below.

\section{Finite element spaces}\label{sec:Discretization}
This section introduces the discrete spaces $U_h$ and $V_h$ of finite element type that we use in the minimal residual method \eqref{eq:uhprime}. 
We specify the 
abstract spaces from Section~\ref{sec:VarForm} and \ref{subsec:discreteSetting} by $H \coloneqq L^2(\Omega)$ and $K \coloneqq H^1_0(\Omega)$ with bounded Lipschitz domain $\Omega \subset \mathbb{R}^d$. 
The construction of the finite element spaces build on an underlying partition $\tria$ of the time-space cylinder $Q$.
\subsection{Partition}\label{subsec:tria}
Our partition $\mathcal{T}$ of $Q \subset \mathbb{R}^{d+1}$ consists of non-overlapping time-space cylinders $K = K_t \times K_x$, where $K_t \subset J$ are time intervals and $K_x \subset \Omega$ are simplices. Such partitions have been shown to be better suited for parabolic problems than purely simplicial meshes, see \cite{DieningStorn22,GantnerStevenson23}. Furthermore, these partitions facilitate parabolically scaled mesh refinements, as motivated by Lemma~\ref{lem:ParaPoincareH12} below. Similar to the case of simplicial meshes, we must impose a set of assumptions to ensure conformity and shape regularity within the mesh:
\begin{assumption}[Partition]\label{ass:Partition} 
\ 
\begin{enumerate}
\item The partition $\tria$ satisfies $\overline{Q} = \bigcup_{K\in \tria} K$, and for any two cells $K_1,K_2\in \tria$ we have $\meas_{d+1}(K_1 \cap K_2)=0$ or $K_1 = K_2$.
\item The elements $K$ are time-space cylinders $K = K_t \times K_x$ with time interval $K_t \subset \overline{\mathcal{J}}$ and simplex $K_x \subset \overline{\Omega}$, where $K_x$ is uniformly shape regular in the sense that $\diam(K_x)^d\eqsim  \meas_{d}(K_x)$.
\item If we have $\meas_{d}(K_1 \cap K_2)>0$ for $K_1 \neq K_2 \in \tria$, then  $K_1 \cap K_2$ is a hyperface of $K_1$ or $K_2$. 
Such $K_1$ and $K_2$ will be called either temporal (hyperface is perpendicular to the time-axis) or spatial neighbors (hyperface is perpendicular to $\lbrace 0 \rbrace \times \mathbb{R}^d$). \label{itm:conformity}
\item We assume for all $K = K_t\times K_x \in \tria$ the parabolic scaling 
\begin{equation}\label{eq:ParabolicScaling}
h_{K,t} \coloneqq \diam(K_t) \eqsim h_{K,x}^2 \coloneqq \diam(K_x)^2.
\end{equation}
\item Neighbors $K_1, K_2 \in \tria$, i.e.~$K_1\cap K_2\neq \emptyset$, are of equivalent size, that is, 
\begin{equation}\label{eq:EqualSizeNeighbors}
h_{K_1,t} \eqsim h_{K_2,t}\qquad\text{and}\qquad h_{K_1,x} \eqsim h_{K_2,x}.
\end{equation}
\item Temporal neighbors $K_1,K_2$ of $K$ that are at the same side of $K$, i.e. $K_t \cap K_{1,t} \cap K_{2,t} \neq \emptyset$, have the same time interval, i.e.~$K_{1,t} = K_{2,t}$.
\label{itm:TimeSpaceCylNeighbors}
\end{enumerate}
\end{assumption}

The assumption in \ref{itm:TimeSpaceCylNeighbors} could be avoided but simplifies the presentation of the results.
We obtain such triangulations with the following (adaptive) refinement strategy, already used in \cite{GantnerStevenson23}:
Our initial partition is a tensor product mesh 
$\tria_0 = \tria_t \otimes \tria_x$ with equidistant partition $\tria_t$ of $\mathcal{J}$ and conforming (in the sense of Ciarlet) triangulation $\tria_x$ of $\Omega$ into simplices. We define the level $\level(K) \coloneqq 0$ of each time-space cell $K\in \tria_0$. Let us assume that there exists a uniform refinement routine for $\tria_x$ leading to conforming partitions with halved mesh sizes, e.g.~the red-refinement routine for $d=2$ or the Maubach-Traxler routine with initialization as discussed in \cite{DieningGehringStorn23} for $d\geq 2$.
Then our local refinement routine $\refine(K)$ for $K = K_t \times K_x$ applies the spatial refinement routine to $K_x$ and splits the time interval $K_t$ into four novel equidistant time intervals. The level of each time-space cell resulting from an application of $\refine(K)$ is set as $\level(K)+1$. This allows us to perform adaptive mesh refinements with a mesh closure routine ensuring that levels of neighboring time-space cells differ at most by one, that is,
\begin{equation}\label{eq:LevelDist}
|\level(K_1) - \level(K_2)|\leq 1\qquad\text{for all }K_1,K_2\in \tria\text{ with }K_1\cap K_2 \neq \emptyset.
\end{equation}

Due to the cylindrical structure of our partitions $\tria$, we obtain for each fixed point $x\in \Omega$ that does not belong to the boundary of some simplex $K_x$ a partition of the time interval $\mathcal{J}$. More precisely, for $x \in \Omega$ with $x \not\in\bigcup \lbrace\partial K_x\colon K = K_t\times K_x  \in \tria \rbrace$ we set the partition of $\cJ$ into essentially non-overlapping subintervals
\begin{equation*}
\tria_t(x) \coloneqq \lbrace K_t \colon K = K_t\times K_x \in \tria\text{ with }x\in K_x\rbrace.
\end{equation*}
For $x \in \bigcup \lbrace\partial K_x\colon K = K_t\times K_x  \in \tria \rbrace \subset \Omega$, being a set of measure zero, we set $\tria_t(x)\coloneqq \emptyset$.
For each $K = K_t \times K_x \in \tria$ we define
\begin{equation}\label{eq:def_qKt}
q_K^t \coloneqq \textup{clos}\Big( \bigcup_{x\in K_x} q_{K_t}(x) \times \lbrace x\rbrace \Big)\ \text{ with } q_{K_t}(x) \coloneqq \bigcup \lbrace K_t' \in \tria_t(x)\colon K_t'\cap K_t \neq \emptyset\rbrace.
\end{equation}
Due to the assumption in \ref{itm:TimeSpaceCylNeighbors} the set $q_K^t = q_{K,t}^t \times K_x$ is a time-space cylinder with time interval $q_{K,t}^t$ and  we can equivalently characterize $q_K^t$ as the union of $K_t$ and the time intervals of the temporal neighbors of $K$ times $K_x$. 
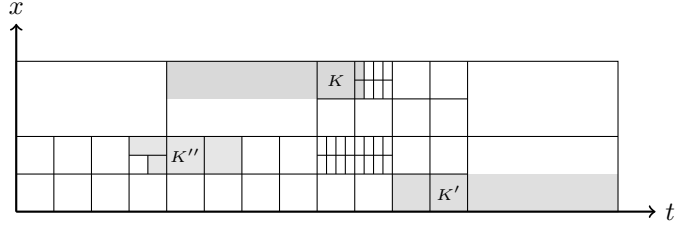
\begin{figure}
\begin{tikzpicture}
     \fill[gray!30] (2,1.5) rectangle (4.625,2); 
     \fill[gray!25] (5,0.5) rectangle (8,0);
     \fill[gray!20] (1.75,0.5) rectangle (3,1);
     \fill[gray!20] (1.5,0.75) rectangle (2,1);
    \draw (0,0) rectangle (8,2); 
    
    \draw (0,1) -- (8,1); 
    
    \draw (2,0) -- (2,2);
    \draw (4,0) -- (4,2);
    \draw (6,0) -- (6,2);
    
    
	\draw(.5,0) -- (.5,1); \draw(1,0) -- (1,1); \draw(1.5,0) -- (1.5,1);    
	\draw(1.5,0.75) -- (2.0,0.75);
    \draw(1.75,0.5) -- (1.75,.75);
    
	\draw (0,.5) -- (6,.5);
	\draw(2.5,0) -- (2.5,1); \draw(3,0) -- (3,1); \draw(3.5,0) -- (3.5,1);
	\draw(4.5,0) -- (4.5,1); \draw(5,0) -- (5,1); \draw(5.5,0) -- (5.5,1);
	\draw (4,1.5) -- (6,1.5);
	\draw(4.5,1) -- (4.5,2); \draw(5,1) -- (5,2); \draw(5.5,1) -- (5.5,2);
	\draw(4,.75) -- (5,.75);	
	\draw(4.25,.5) -- (4.25,1); \draw(4.125,.5) -- (4.125,1); \draw(4.375,.5) -- (4.375,1);
	\draw(4.75,.5) -- (4.75,1); \draw(4.625,.5) -- (4.625,1); \draw(4.875,.5) -- (4.875,1);
	
	\draw(4.5,1.75) -- (5,1.75);
	\draw(4.75,1.5) -- (4.75,2); \draw(4.625,1.5) -- (4.625,2); \draw(4.875,1.5) -- (4.875,2);
    
	\node at (4.25,1.75) {\tiny $K$};
	\node at (5.75,0.25) {\tiny $K'$};
	\node at (2.25,0.75) {\tiny $K''$};
    
    \draw[->,thick] (0,0) -- (0,2.5) node[anchor=south] {$x$};
    \draw[->,thick] (0,0) -- (8.5,0) node[anchor=west] {$t$};
\end{tikzpicture}
\caption{A partition of a time-space cylinder with $d=1$, where the areas $q_K^t$, $q_{K'}^t$, and $q_{K''}^t$ for time-space cells $K,K',K''$ are marked gray. 
Notice that Assumption~\ref{ass:Partition}~\ref{itm:TimeSpaceCylNeighbors} does \emph{not} permit the area $q_{K''}^t$, since the left neighbors of $K''$ do not have the same time interval.}\label{fig:examplePartition}
\end{figure}
The grading assumption in \eqref{eq:EqualSizeNeighbors} allows for the following localization.
\begin{lemma}[Localized norm]\label{lem:localizedGeneral}
Each $v\in H^{1/2}(\mathcal{J};L^2(\Omega))$ satisfies 
\begin{equation}
\lvert v \rvert_{H^{1/2}(\cJ;L^2(\Omega))}^2 \lesssim \sum_{K\in \tria} \lvert v \rvert_{H^{1/2}(q_{K,t}^t;L^2(K_x))}^2 + h_{K,t}^{-1}\, \lVert v \rVert^2_{L^2(K)}.
\end{equation}
\end{lemma}
\begin{proof}
Let $v\in H^{1/2}(\mathcal{J};L^2(\Omega))$, and let $x \in \Omega$ with $\tria_t(x) \neq \emptyset$. Then $w \coloneqq v(\bigcdot,x)$ satisfies \cite[Lem.~4.2]{CarstensenFaermann01}
\begin{align*}
&\int_\mathcal{J} \int_\mathcal{J} \frac{|w(t)-w(s)|^2}{|t-s|^2} \dt\ds\\
&\qquad  \lesssim \sum_{K_t\in \tria_t(x)}  \int_{q_{K_t}(x)}  \int_{q_{K_t}(x)} \frac{|w(t)-w(s)|^2}{|t-s|^2} \dt\ds + |K_t|^{-1} \int_{K_t} |w(t)|^2 \dt.
\end{align*}
Integrating over all $x\in \Omega$ concludes the proof.
\end{proof}
\subsection{Finite element spaces}
Let $\tria$ be a partition of $Q$ satisfying Assumption~\ref{ass:Partition}.
For each time-space cylinder $K = K_t\times K_x\in \tria$ and all $p = (p_x,p_t)\in \mathbb{N}^2_0$ we set the polynomial spaces
\begin{align*}
\mathbb{P}_{p_t}(K_t) & \coloneqq \lbrace v_x \in L^2(K_t) \colon v_t \text{ is a polynomial of maximal degree }p_t\rbrace,\\
\mathbb{P}_{p_x}(K_x) & \coloneqq \lbrace v_x \in L^2(K_x) \colon v_x \text{ is a polynomial of maximal degree }p_x\rbrace,\\
\mathbb{P}_p(K) & \coloneqq \mathbb{P}_{p_t}(K_t) \otimes \mathbb{P}_{p_x}(K_x) \coloneqq \textup{span}\lbrace v_tv_x \colon v_t \in \mathbb{P}_{p_t}(K_t)\text{ and } v_x \in \mathbb{P}_{p_x}(K_t)\rbrace.
\end{align*}
Moreover, we set for all $p = (p_t,p_x)\in \mathbb{N}^2$ the space of piece-wise polynomial and globally continuous functions
\begin{align*}
W_h^{p_t,p_x} \coloneqq \lbrace v \in C^0(Q) \colon v|_K \in \mathbb{P}_p(K)\text{ for all }K\in \tria\rbrace.
\end{align*}
To include the homogeneous boundary data on $\mathcal{J} \times \partial \Omega$ we set 
\begin{equation}\label{eq:DefW_h0ptpx}
W_{h,0}^{p_t,p_x} \coloneqq W_h^{p_t,p_x} \cap W_0 \qquad\text{with } W_0 \coloneqq L^2(\mathcal{J} ;H^1_0(\Omega)) \cap H^{1/2}(\cJ;L^2(\Omega)).
\end{equation}
Our discrete trial space then reads with $p = (p_t,p_x) \in \mathbb{N}^2$
\begin{equation*}
U_h \coloneqq W_{h,0}^{p_t,p_x}.
\end{equation*}
The test space $V_h$ will be the sum of the lowest-order space 
\begin{equation*}
V_{h,0} \coloneqq W_{h,0}^{1,1} \cap V
\end{equation*}
and the span of suitable bubble functions defined as follows:

For all $K = K_t \times K_x \in \tria$ we set the non-negative bubble function
\begin{equation*}
\mathcal{B}_K \in \mathbb{P}_2(K_t) \otimes \mathbb{P}_{d+1}(K_x)\qquad\text{ with } \mathcal{B}_K|_{\partial K} = 0\text{ and } \lVert \mathcal{B}_K \rVert_{L^2(K)} = 1.
\end{equation*}
In other words, $\mathcal{B}_K$ is the scaled product of all local lowest-order basis functions associated to the vertices of $K$. We define the space of volume bubble functions as 
\begin{equation} \label{eq:volumebubble}
V_{h,2} \coloneqq \bigoplus_{K\in \tria} \mathcal{B}_K \big(\mathbb{P}_{p_t}(K_t) \otimes \mathbb{P}_{p_x-2}(K_x)\big) + \mathcal{B}_K \big(\mathbb{P}_{p_t-1}(K_t) \otimes \mathbb{P}_{p_x}(K_x)\big).
\end{equation}
Additionally, we define the following face bubble functions. 
Let $\mathcal{F}$ denote the set of all faces ($d$-hyperfaces of time-space cells) in $\tria$. 
We denote by $\mathcal{F}_t$ the set of all faces $f\in \mathcal{F}$ in $\tria$ that are orthogonal to the $\{0\} \times \Omega$ plane, and that do not contain children, that is,
\begin{equation*}
\mathcal{F}_t \coloneqq \lbrace f \in \mathcal{F} \colon f \perp \lbrace 0 \rbrace \times \mathbb{R}^d\text{ and any } g\in \mathcal{F}\text{ with } g \subset f \text{ satisfies }g =f  \rbrace.
\end{equation*}
The set $\mathcal{F}_{t,0} \subset \mathcal{F}_t$  denotes the set of all faces $f\in \mathcal{F}_t$ that are not on the lateral boundary, that is, $f\not \subset \overline{\cJ} \times \partial\Omega$.
We set for all faces $f \in \mathcal{F}_{t,0}$ the non-negative face bubble function
\begin{equation*}
\mathcal{B}_f \in W^{2,d}_h\qquad\text{with } \lVert \mathcal{B}_f \rVert_{L^2(Q)}=1\text{ and }\mathcal{B}_f|_g = 0\text{ for all }g\in \mathcal{F} \text{ with }g\not\subset f.
\end{equation*}
In other words, $\mathcal{B}_f$ is the scaled product of the local lowest-order basis functions associated to the vertices of $f$. Let $(\psi_j)_{j\in \mathcal{N}^{p_t,p_x-1}} \subset W^{p_t,p_x-1}_h$ denote the nodal basis with associated Lagrange nodes $\mathcal{N}^{p_t,p_x-1} \subset \overline{Q}$. Then we set the space of face bubble functions as 
\begin{equation} \label{eq:facebubble}
V_{h,1} \coloneqq \bigoplus_{f\in \mathcal{F}_{t,0}} \textup{span} \lbrace \mathcal{B}_f\, \psi_j \colon j \in \mathcal{N}^{p_t,p_x-1} \cap f  \rbrace.
\end{equation}
We finally set our discrete test space 
\begin{equation}\label{eq:testSpace}
V_h \coloneqq V_{h,0} \oplus V_{h,1} \oplus V_{h,2}.
\end{equation}
\begin{remark}[Optimal polynomial degrees] \label{rem:1} 
Let $\tria$ be a quasi-uniform triangulation with mesh sizes $h_x$ in space and $h_t$ in time. For smooth functions $u \in W_0$, the best approximation error $\min_{v_h\in W_{h,0}^{p_t,p_x}} \lVert u - v_h \rVert_W$, cf.~\eqref{eq:defNormW} below, is proportional to 
\begin{equation*}
\max\big\lbrace \max\lbrace h_t^{p_t+1},h_x^{p_x}\rbrace,\max \lbrace h_t^{p_t+1/2},h_x^{p_x+1}\rbrace \big\rbrace =\max \lbrace h_t^{p_t+1/2},h_x^{p_x}\rbrace.
\end{equation*}
Using the parabolic scaling $h_t=h_x^2$, both terms are equal when $2 p_t+1=p_x$. 
Since the same result is valid when $W_0$ is replaced by $L^2(\mathcal{J} ;H^1_0(\Omega)) \cap B^{1/2,p,2}(\mathcal{J};L^2(\Omega))$, Proposition~\ref{prop:2} shows that this result is also valid for the space $U$. Consequently, the choice $2 p_t+1=p_x$ of the  polynomial degrees $p=(p_t,p_x)$ for the discrete trial space $U_h$  is the most efficient one.
\end{remark}

\section{Interpolation and Fortin operator} \label{sec:Interpol}
The quasi-best approximation property \eqref{eq:QuasiBest} of our numerical scheme \eqref{eq:uhprime} requires the existence of a Fortin operator. We design such an operator in Section~\ref{subsec:Fortin}. A key tool in its design is a stable interpolation operator, which we analyze in the following Section~\ref{subsec:FEspaces}.
\subsection{Interpolation operator}\label{subsec:FEspaces}
We have shown in Lemma~\ref{lem:localizedGeneral} that the global norm of $v\in W$ is bounded from above by the sum of localized components. Equivalence does in general not hold. We will show, however, that local stable projectors $\mathcal{I} \colon L^2(\cJ;H^1(\Omega)) \to W_{h}^p$ with $p = (p_t,p_x) \in \mathbb{N}^2$ allow for a localization of the interpolation error.
We thereby focus on a Scott--Zhang operator similar to the original one in \cite{ScottZhang90}. Its definition involves the $L^2$ inner product $\langle \bigcdot ,\bigcdot \rangle_q \coloneqq \int_q \bigcdot \cdot \bigcdot \dx$ for volume contributions $q \subset \overline{Q}$ or boundary contributions $q\subset \overline{\mathcal{J}} \times \partial \Omega$.
\begin{definition}[Scott--Zhang operator]\label{def:SZ}
Let $p = (p_t,p_x) \in \mathbb{N}^2$ and let $(\psi_{j})_{j \in \mathcal{N}^p}$ denote the nodal basis of $W_h^p$ with corresponding Lagrange nodes $\mathcal{N}^p$. We assign to each nodal basis function $\psi_j$ with $j\not\in \overline{\cJ}\times \partial \Omega$ a time-space cylinder $S_j \in \tria$. If the Lagrange node of $\psi_\ell$ is on the lateral boundary $\overline{\cJ} \times \partial \Omega$, we choose $S_j \in \faces_t$ with $j\in S_j$ and define the bi-orthogonal basis function $\psi_j^* \in \mathbb{P}_{p}(S_j)$ such that
\begin{equation}\label{eq:DualWeight}
\langle \psi_\ell,\psi_j^*\rangle_{S_j} = \delta_{j,\ell}\qquad\text{for all }\ell \in \mathcal{N}^p(S_j),   
\end{equation}
where $\mathcal{N}^p(S_j)$ denotes the set of Lagrange nodes for $\mathbb{P}_{p}(S_j)$.
We set the operator 
\begin{equation}\label{eq:DefI}
\mathcal{I} \coloneqq \sum_{j\in \mathcal{N}^p} \langle \bigcdot ,\psi_{j}^*\rangle_{S_j} \psi_j.
\end{equation}
\end{definition}

Classical scaling arguments as for example in \cite{ScottZhang90,BrennerScott08} imply the following.
\begin{lemma}[Basic properties of $\mathcal{I}$] \label{lem:basicpropsSZ}
Any Scott--Zhang-type operator $\mathcal{I}$ defined as in Definition~\ref{def:SZ} 
\begin{itemize}
\item preserves homogeneous boundary data in the sense that $\mathcal{I}\colon W_0 \to W_{h,0}^p$,
\item satisfies the projection property $\mathcal{I} w_h = w_h$ for all $w_h \in W^p_h$,
\item is local in the sense that one has for all $K\in \tria$ the identity $(\mathcal{I} v)|_{K} = (\mathcal{I} w)|_{K}$ for all $v,w \in L^2(\cJ;H^1(\Omega))$ with $v|_{q_k} = w|_{q_k}$ and element patch 
\begin{equation*}
q_K \coloneqq \bigcup \lbrace \textup{supp}(\psi_j)\colon j \in \mathcal{N}^p\text{ and }K\subset \textup{supp}(\psi_j)\rbrace,
\end{equation*}
\item is locally  $L^2$ stable on $W_0$ in the sense that
\begin{equation}\label{eq:LocL2Stab}
\lVert \mathcal{I} w \rVert_{L^2(K)} \lesssim \lVert w \rVert_{L^2(q_K)}\qquad\text{for all }K\in \tria \text{ and }w\in W_{0}.
\end{equation}
\end{itemize}
\end{lemma}
Before we can verify the approximation properties of $\mathcal{I}$, we introduce the following version of \Poincare's inequality adapted to the spaces at hand.
Given a set $q \subset \mathbb{R}^m$ with $m\in \mathbb{N}$ and $m$-dimensional Lebesgue measure $|q|_m$, we define the integral mean of a function $g\colon q \to \mathbb{R}$ by
\begin{align*}
\langle g \rangle_q \coloneqq \dashint_q g \dx \coloneqq \frac{1}{|q|_m} \int_q g\dx.
\end{align*}

\begin{lemma}[Parabolic \Poincare{}]\label{lem:ParaPoincareH12}
Let $v\in L^2(K_t;H^1(K_x))\cap H^{1/2}(K_t;L^2(K_x))$ with time-space cylinder $K = K_t \times K_x$, where $K_t$ is an interval of length $h_t$ and $K_x \subset \mathbb{R}^d$ is some bounded domain with diameter $h_x$. We assume that $K_x$ allows for the \Poincare{} estimate $\lVert p - \langle p \rangle_{K_x}\rVert_{L^2(K_x)} \leq C_x h_x \lVert \nabla_x p \rVert_{L^2(K_x)}$ with constant $C_x < \infty$  for all $p\in H^1(K_x)$. Then we have
\begin{align*}
\lVert v - \langle v \rangle_K \rVert^2 _{L^2(K)}\leq C_x^2  h^2_x\, \lVert \nabla_x v \rVert^2_{L^2(K)} + h_t\, \lvert v \rvert^2_{H^{1/2}(K_t;L^2(K_x))}.
\end{align*}
\end{lemma}
\begin{proof}
Let $v\in L^2(K_t;H^1(K_x))\cap H^{1/2}(K_t;L^2(K_x))$ with time-space cylinder $K=K_t\times K_x$. The Pythagorean theorem and Fubini's theorem yield
\begin{align}\label{eq:Triangle1}
\lVert v - \langle v \rangle_K \rVert_{L^2(K)}^2 = \lVert v - \langle v \rangle_{K_x} \rVert_{L^2(K)}^2 
+ \lVert \langle v - \langle v \rangle_{K_t} \rangle_{K_x}\rVert^2_{L^2(K)}.
\end{align}
The first term is bounded due to \Poincare's inequality by 
\begin{align*}
\lVert v - \langle v \rangle_{K_x} \rVert^2_{L^2(K)} \leq C_x^2 h^2_x\, \lVert \nabla_x v \rVert_{L^2(K)}^2.
\end{align*}
Using Jensen's inequality for the squared second term in \eqref{eq:Triangle1} leads to 
\begin{align*}
 \lVert \langle v - \langle v \rangle_{K_t} \rangle_{K_x}\rVert^2_{L^2(K)} & = \int_K \Big|\dashint_{K_x} \dashint_{K_t} v(t,x) - v(s,x) \ds\dx\Big|^2 \mathrm{d}(t,y)\\
 & \leq \int_K  \dashint_{K_x} \dashint_{K_t} |v(t,x) - v(s,x)|^2 \ds\dx\, \mathrm{d}(t,y)\\
& = \int_{K_t} \dashint_{K_t} \lVert v(t) - v(s)\rVert_{L^2(K_x)}^2 \ds\dt\\
 &\leq h_t \int_{K_t} \int_{K_t} \frac{\lVert v(t) - v(s)\rVert_{L^2(K_x)}^2}{|t-s|^2} \ds\dt.
\end{align*}
Combining these results concludes the proof.
\end{proof}
Adaptively refined meshes require the use of the parabolic \Poincare{} inequality for non-cylindrical areas $q \subset \overline{Q}$. In our partitions these areas are build by unions of overlapping time-space cylinders $q_0,\dots,q_L$ with $L\in \mathbb{N}$ in the sense that
\begin{equation}\label{eq:decompQ}
q = \bigcup_{\ell=0}^L q_\ell\quad\text{and}\quad
|q_j| \eqsim \left|\bigcup_{\ell=0}^{j-1} q_\ell \cap q_{j}\right|\qquad \text{for all }j=1,\dots,L.
\end{equation} 
The generalized parabolic \Poincare{} inequality uses for all $v\in W$ with 
\begin{equation*}
W \coloneqq L^2(\cJ;H^1(\Omega)) \cap H^{1/2}(\cJ;L^2(\Omega))
\end{equation*}
 the semi-norm  
\begin{equation*}
\lvert v \rvert^2_{H^{1/2}L^2(q)} \coloneqq \int_\mathcal{J} \int_\mathcal{J} \int_\Omega \indicator_q(x,t) \indicator_q(x,s) \frac{|v(x,t) - v(x,s)|^2}{|t-s|^2} \dx \ds \dt.
\end{equation*}
\begin{lemma}[Generalized parabolic \Poincare{}]\label{lem:GenParaPoincareH12}
Suppose that $q \subset Q$ is some area covered by time-space cylinders $q_1,\dots,q_L$ as in \eqref{eq:decompQ}. Let $h_t$ and $h_x$ denote the diameters of $q$ in time and space direction, respectively.
Then we have for all $v\in W$ the parabolic \Poincare{} inequality
\begin{equation*}
\lVert v - \langle v \rangle_q \rVert^2_{L^2(q)} \lesssim  h_x^2 \, \lVert \nabla_x v \rVert_{L^2(q)}^2 + h_t\, \lvert v \rvert^2_{H^{1/2}L^2(q)}.
\end{equation*}
The hidden constant depends on $L$, the hidden constant in \eqref{eq:decompQ}, and the \Poincare{} constant in space $C_x$ for each time-space cylinder.
\end{lemma}
\begin{proof}
Let $v\in W$ and suppose $q \subset Q$ has a decomposition $q_0,\dots,q_L$ as in \eqref{eq:decompQ}. It holds, with $q^- \coloneqq  \bigcup_{\ell=0}^{L-1} q_\ell$, 
\begin{align*}
\lVert v - \langle v \rangle_q \rVert_{L^2(q)}^2 \leq  \lVert v - \langle v \rangle_{q^-} \rVert_{L^2(q)}^2 \leq \lVert v - \langle v \rangle_{q^-} \rVert_{L^2(q_L)}^2 + \lVert v - \langle v \rangle_{q^-} \rVert_{L^2(q^-)}^2.
\end{align*}
The first term satisfies 
\begin{align*}
\lVert v - \langle v \rangle_{q^-} \rVert_{L^2(q_L)}^2 = \lVert v - \langle v \rangle_{q_L} \rVert_{L^2(q_L)}^2 + \lVert \langle v \rangle_{q^-} - \langle v \rangle_{q_L} \rVert_{L^2(q_L)}^2.
\end{align*}
Jensen's inequality and \eqref{eq:decompQ} yield for the second addend, with $q^\textup{int}_L \coloneqq q^-\cap q_L$,
\begin{align*}
&\lVert \langle v \rangle_{q^-} - \langle v \rangle_{q_L} \rVert_{L^2(q_L)}^2 \leq 2 \, \lVert \langle v \rangle_{q^-} - \langle v \rangle_{q^\textup{int}_L} \rVert_{L^2(q_L)}^2+2\, \lVert \langle v \rangle_{q^\textup{int}_L} - \langle v \rangle_{q_L} \rVert_{L^2(q_L)}^2 \\
& \qquad = 2 \int_{q_L} \Big| \dashint_{q_L^\textup{int}} v  - \langle v \rangle_{q^-}\dy \Big|^2 \dz + 2 \int_{q_L} \Big| \dashint_{q_L^\textup{int}} v  - \langle v \rangle_{q_L}\dy \Big|^2 \dz \\
&\qquad \leq 2 \int_{q_L} \dashint_{q^\textup{int}_L} |v -\langle v \rangle_{q^-}|^2 \dy\dz + 2 \int_{q_L} \dashint_{q^\textup{int}_L} |v -\langle v \rangle_{q_L}|^2 \dy\dz\\
&\qquad = \frac{2\, |q_L|}{|q_L^\textup{int}|} \big( \lVert v -\langle v \rangle_{q^-} \rVert_{L^2(q^\textup{int}_L)}^2 + \lVert v -\langle v \rangle_{q_L} \rVert_{L^2(q_L)}^2\big)\\
&\qquad\leq \frac{2\, |q_L|}{|q_L^\textup{int}|} \big( \lVert v -\langle v \rangle_{q^-} \rVert_{L^2(q^-)}^2 + \lVert v -\langle v \rangle_{q_L} \rVert_{L^2(q_L)}^2\big).
\end{align*}
The contributions $\lVert v -\langle v \rangle_{q_L} \rVert_{L^2(q_L)}^2$ can be bounded by Lemma~\ref{lem:ParaPoincareH12}. We proceed inductively to bound the term $\lVert v -\langle v \rangle_{q^-} \rVert_{L^2(q^-)}$.
\end{proof}
\begin{remark}[Boundary]\label{rem:BoundaryPoincare}
Suppose that $q_L$ in \eqref{eq:decompQ} touches the lateral boundary of $Q$ in the sense that $q_L \cap \cJ \times \partial \Omega \neq \emptyset$. For any $v\in W_0$ one has the identity
\begin{align*}
\lVert v \rVert^2_{L^2(q)} = \lVert v - \langle v\rangle_q \rVert^2_{L^2(q)} + \lVert \langle v\rangle_q \rVert_{L^2(q)}^2.
\end{align*}
The first term can be bounded by Lemma~\ref{lem:GenParaPoincareH12}. Similar arguments as in the previous proof bound the second term by
\begin{align*}
\lVert \langle v\rangle_q \rVert_{L^2(q)} \leq \lVert \langle v\rangle_q - \langle v\rangle_{q_L^\textup{int}} \rVert_{L^2(q)} + \lVert \langle v\rangle_{q_L^\textup{int}} \rVert_{L^2(q)} \lesssim \lVert v - \langle v\rangle_q \rVert_{L^2(q)} + \lVert v \rVert_{L^2(q_L)}.
\end{align*}
Since \Poincare's inequality yields $\lVert v \rVert_{L^2(q_L)} \lesssim h_{K,x} \lVert \nabla_x v\rVert_{L^2(q_L)}$, we obtain
\begin{align*}
\lVert v  \rVert^2_{L^2(q)} \lesssim  h_x^2 \, \lVert \nabla_x v \rVert_{L^2(q)}^2 + h_t\, \lvert v \rvert^2_{H^{1/2}L^2(q)}.
\end{align*}
\end{remark}
Combining the properties of the Scott--Zhang operator with the parabolic \Poincare{} inequality leads to the following approximation result. 
\begin{lemma}[Local approximation property]\label{lem:locApx}
We have for all $K\in \tria$ 
and $w \in W_0$
\begin{align*}
\lVert w - \mathcal{I} w \rVert_{L^2(K)} &\lesssim h_{K,x} \lVert \nabla_x w \rVert_{L^2(q_K)} + h_{K,t}^{1/2}\, \lvert w \rvert_{H^{1/2}L^2(q_K)},\\
\lVert w - \mathcal{I} w \rVert_{L^2(K)} &\lesssim h_{K,x} \lVert \nabla_x (w - \mathcal{I} w) \rVert_{L^2(q_K)} + h_{K,t}^{1/2}\, \lvert w- \mathcal{I} w \rvert_{H^{1/2}L^2(q_K)}.
\end{align*}
\end{lemma}
\begin{proof}
Let $K\in \tria$ and $w \in W_0$.
By Assumption~\ref{ass:Partition} there exists a decomposition of $q = q_K$ into time-space cylinders $q_1,\dots,q_L$ as in \eqref{eq:decompQ}. If the patch $q_K$ touches the lateral boundary of $Q$ in the sense that $q_K \cap \cJ\times \partial \Omega \neq \emptyset$, Remark~\ref{rem:BoundaryPoincare} and the local $L^2$ stability of $\mathcal{I}$ for functions in $W_0$ yield
\begin{align*}
&\lVert w - \mathcal{I} w \rVert_{L^2(K)} \leq \lVert w \rVert_{L^2(K)} + \lVert \mathcal{I} w \rVert_{L^2(K)} \lesssim \lVert w \rVert_{L^2(q_K)}\\
&\qquad \lesssim h_{K,x} \lVert \nabla_x w \rVert_{L^2(q_K)} + h_{K,t}^{1/2} \lvert w\rvert_{H^{1/2}L^2(q_K)}.
\end{align*}
If $q_K \cap \cJ\times \partial \Omega = \emptyset$, the identity $\mathcal{I}1 = 1$, the triangle inequality, the local $L^2$ stability of $\mathcal{I}$ for functions in $W_0$, and Lemma~\ref{lem:GenParaPoincareH12} yield
\begin{align*}
&\lVert w - \mathcal{I} w \rVert_{L^2(K)} \leq \lVert w - \langle w \rangle_{q_K}\rVert_{L^2(K)} + \lVert \mathcal{I} (w-\langle w \rangle_{q_K})\rVert_{L^2(K)} \lesssim \lVert w-\langle w \rangle_{q_K} \rVert_{L^2(q_K)}\\
&\qquad \lesssim h_{K,x} \lVert \nabla_x w \rVert_{L^2(q_K)} + h_{K,t}^{1/2} \lvert w\rvert_{H^{1/2}L^2(q_K)}.
\end{align*}
This proves the first estimate.
 
The second estimate follows by the projection property allowing us to replace $w$ by $w-\mathcal{I} w$ in the first estimate of the lemma.
\end{proof}
Recall the definition of $q_K^t$ in \eqref{eq:def_qKt}.
Combining Lemma~\ref{lem:localizedGeneral} and \ref{lem:locApx} with the finite overlap of patches and the parabolic scaling yields the following localization of the norm in $W$ defined by
\begin{equation}\label{eq:defNormW}
\lVert \bigcdot \rVert_W \coloneqq (\lVert \nabla_x \bigcdot \rVert^2_{L^2(Q)} +  \lvert \bigcdot \rvert_{H^{1/2}(\cJ;L^2(\Omega))}^2)^{1/2}.
\end{equation}
\begin{theorem}[Localization of interpolation error]\label{thm:LocalizationIIold}
The interpolation error $\delta \coloneqq w - \mathcal{I}w$ with $w\in W_0$ localizes in the sense that 
\begin{align*}
\lVert \delta \rVert_W^2& \eqsim \sum_{K\in \tria}\lVert  \nabla_x \delta \rVert_{L^2(K)}^2 + \lvert  \delta \rvert_{H^{1/2}L^2(q^t_{K})}^2 + h_{K,t}^{-1} \lVert \delta \rVert_{L^2(K)}^2\\
& \eqsim \sum_{K\in \tria}\lVert  \nabla_x \delta \rVert_{L^2(K)}^2 + \lvert  \delta \rvert_{H^{1/2}L^2(q_K)}^2.
\end{align*}
\end{theorem}
In order to obtain local best-approximation properties of $\mathcal{I}$, we introduce the following inverse estimate.
\begin{lemma}[Inverse estimate]\label{lem:inverseEst}
For any $v_h \in W_h^p$ and $K\in \tria$ one has the inverse estimate, with hidden constant depending only on the polynomial degree $p_t$ and the equivalence constants in \eqref{eq:EqualSizeNeighbors}, 
\begin{equation*}
\lvert v_h \rvert_{H^{1/2}L^2(q_{K}^t)} \lesssim h_{K,t}^{-1/2}\, \lVert v_h\rVert_{L^2(q_K^t)}.
\end{equation*}  
\end{lemma}
\begin{proof}
Let $v_h \in W_h^p$ and let $x\in K_x$ with $K=K_t\times K_x\in \tria$. The function $v_h(\bigcdot,x)$ is a piece-wise polynomial with respect to the quasi-uniform (see~\eqref{eq:EqualSizeNeighbors}) underlying partition $\lbrace K_{t,1},K_{t,2},K_{t,3} \rbrace$ of $q_{K_t}(x)$. Scaling arguments yield
\begin{align*}
\lvert v_h(\bigcdot,x) \rvert_{H^{1/2}(q_{K_t}(x))}^2 &= \sum_{j = 1}^3 \sum_{k=1}^3 \int_{K_{t,j}} \int_{K_{t,k}} \frac{|v_h(t,x) -v_h(s,x) |^2}{|t-s|^2} \dt \ds \\
&\lesssim h_{K,t}^{-1} \, \lVert v_h(\bigcdot,x)  \rVert_{L^2(q_{K_t}(x))}^2.
\end{align*}
Integrating over all $x\in K_x$ concludes the proof.
\end{proof}
With the inverse estimate and the local $L^2$-stability of $\mathcal{I}$ we can bound the addends of the localized interpolation error by best-approximations as shown in the following lemma. The result involves the extended patch 
\begin{equation*}
q(q_K^t) \coloneqq \bigcup \lbrace q_{K'} \colon K'\in \tria\text{ with }\textup{int}(K') \cap q_K^t \neq \emptyset\rbrace.
\end{equation*}
\begin{lemma}[Local quasi-optimality]\label{lem:LocalQuasiOpt}
The interpolation error $\delta \coloneqq w - \mathcal{I} w$ with $w\in W_0$ satisfies on each $K\in \tria$ the local quasi-optimality result
\begin{align*}
&\lVert  \nabla_x \delta \rVert_{L^2(K)}^2 + \lvert  \delta \rvert_{H^{1/2}L^2(q^t_K))}^2 + h_{K,t}^{-1} \lVert \delta \rVert_{L^2(K)}^2 \\
& \qquad\quad  \lesssim \min_{w_h \in W^p_{h,0}} \lVert  \nabla_x (w-w_h) \rVert_{L^2(q(q_K^t))}^2 + \lvert w - w_h \rvert^2_{H^{1/2}L^2(q(q_K^t))}.
\end{align*}
\end{lemma}
\begin{proof}
Let $w \in W_0$ and $K\in \tria$. We focus on deriving an upper bound for the term $\lvert  \delta \rvert_{H^{1/2}L^2(q^t_K)}^2$. The remaining bounds follow similarly. The triangle inequality and the projection property of $\mathcal{I}$ yield for all $w_h \in W_{h,0}^p$
\begin{equation*}
\lvert \delta \rvert_{H^{1/2}L^2(q^t_K)} \leq  \lvert w- w_h \rvert_{H^{1/2}L^2(q^t_K)} + \lvert \mathcal{I}(w-w_h) \rvert_{H^{1/2}L^2(q^t_K)}.
\end{equation*}
The inverse estimate in Lemma~\ref{lem:inverseEst}, the local $L^2$-stability of $\mathcal{I}$ from Lemma~\ref{lem:basicpropsSZ} for functions in $W_0$, and the parabolic \Poincare{} inequality in Lemma~\ref{lem:GenParaPoincareH12} lead to
\begin{align*}
\lvert \mathcal{I}(w-w_h) \rvert_{H^{1/2}L^2(q^t_K)} & = \lvert \mathcal{I}(w-w_h - \langle w - w_h\rangle_{q(q_K^t)}) \rvert_{H^{1/2}L^2(q^t_K)}\\
& \lesssim h_{t,K}^{-1/2}\lVert  w - w_h - \langle w - w_h\rangle_{q(q_K^t)} \rVert_{L^2(q(q_K^t))} \\
& \lesssim \lVert \nabla_x (w-w_h) \rVert_{L^2(q(q_K^t))} + \lvert w - w_h \rvert_{H^{1/2}L^2(q(q_K^t))}.\qedhere
\end{align*}
\end{proof}
Combining Theorem~\ref{thm:LocalizationIIold} and Lemma~\ref{lem:LocalQuasiOpt} yields the following corollary.
\begin{corollary}[Local best-approximation]\label{cor:QuasBest}
One has for all $w\in W_0$
\begin{align*}
\lVert w-\mathcal{I} w \rVert_W^2
 \eqsim \sum_{K\in \tria} \min_{w_h\in W^p_{h,0}} \lVert  \nabla_x (w-w_h) \rVert_{L^2(q(q^t_K))}^2 + \lvert w-w_h \rvert_{H^{1/2}L^2(q(q^t_K))}^2.
\end{align*} 
\end{corollary}
By taking $w_h=0$ in the corollary and using the finite overlap of patches, we obtain the following stability result.
\begin{corollary}[Stability]\label{cor:Stability}
The operator $\mathcal{I}\colon W_0 \to W_{h,0}^p$ is uniformly stable with respect to the norm in $W$ in the sense that 
\begin{align*}
\|\mathcal{I} u\|_W^2+ \sum_{K \in \tria} h_{K,x}^{-2} \, \|(1-\mathcal{I})u\|^2_{L^2(K)} \lesssim \|u\|_W^2.
\end{align*}
\end{corollary}
We conclude this subsection with a modification of $\mathcal{I}$ mapping into discretizations of $U_F$ or $V_F$ defined in \eqref{eq:DefUvV}.
\begin{remark}[Zero initial or end data]\label{rem:IforUandV}
Let $V_h \coloneqq  W^p_{h,0} \cap V_F$. Then we extend the partition $\tria$ to a partition of the half space $(0,\infty) \times \Omega$ by mirroring the partition at times $T, 2T,\dots$ . We now define the operator $\mathcal{I}$ on these partitions as in Definition~\ref{def:SZ} of the half space such that degrees of freedom $\psi_\ell$ with Lagrange node at time $T$ use dual basis functions $\psi_\ell^*$ supported outside of $Q$. The resulting operator $\mathcal{I}$ maps $V_F$ onto $V_h$ and has the localization property displayed in Theorem~\ref{thm:LocalizationIIold}, where the left-hand side involves the $H^{1/2}((0,\infty);H)$ norm that is equivalent to the $H^{1/2}_{,0}(\cJ;H)$ norm as discussed in Remark~\ref{rem:ExtByZero}.
This allows us to replace the norm $\lVert \bigcdot \rVert_W$ in Lemma~\ref{lem:LocalQuasiOpt} as well as Corollary~\ref{cor:QuasBest} and \ref{cor:Stability} by the norm $\lVert \bigcdot \rVert_V$.
Similar modifications lead to an interpolation operator $\mathcal{I} \colon U_F \to W^p_{h,0}\cap U_F$.
\end{remark}
\subsection{Fortin operator}\label{subsec:Fortin}
In this subsection we extend the design of the interpolation operator in Section~\ref{subsec:FEspaces} to ensure the existence of a bounded operator $F\colon V \to V_h \subset V$ that satisfies, with trial space $U_h \coloneqq W_{h,0}^p$ and polynomial degree $p = (p_t,p_x)\in \mathbb{N}^2$, the annihilation property
\begin{equation*}
b(u_h,v-F v) = 0\qquad\text{for all }u_h \in U_h\text{ and }v\in V.
\end{equation*}
It is known that the existence of such a so-called Fortin operator yields the injectivity of the discretized operator equation $B\colon U_h \to V_h'$ \cite{Fortin77} and leads to quasi-optimality in suitable minimal residual methods \cite[Thm.~3.6]{MonsuurStevensonStorn23}. 

\begin{theorem}[Fortin operator]\label{thm:FortinOperator}
There exists a continuous linear operator $F\colon V \to V_h$ with the annihilation property
\begin{equation*}
b(u_h,v - F v) = 0\qquad\text{for all }u_h\in U_h\text{ and }v\in V.
\end{equation*}
\end{theorem}
\begin{proof}
\textit{Step 1 (Definition of the operator).}
From \eqref{eq:volumebubble} and  \eqref{eq:facebubble}, recall the definition of the bubble spaces $V_{h,1}$ and $V_{h,2}$.
We define the correction operator $\mathcal{C}_1 \colon V \to V_{h,1}$, which is according to the definition of $V_{h,1}$ uniquely characterized by its values on the faces $f\in \mathcal{F}_{t,0}$, such that
\begin{equation}\label{eq:DefC1}
\int_f \xi (v- \mathcal{C}_1 v)\ds = 0 \quad\text{for all }v\in V, f\in \mathcal{F}_{t,0}, \xi \in \mathbb{P}_{p_t}(f_t)\otimes \mathbb{P}_{p_x-1}(f_x).
\end{equation}
Moreover, we set the correction operator $\mathcal{C}_2 \colon V \to V_{h,2}$ such that for all $v\in V$, $K\in \tria$, and $\theta \in \mathbb{P}_{p_t}(K_t) \otimes \mathbb{P}_{p_x-2}(K_x) + \mathbb{P}_{p_t-1}(K_t) \otimes \mathbb{P}_{p_x}(K_x)$
\begin{equation}\label{eq:DefC2}
\int_K \theta (v-\mathcal{C}_2 v) \dx = 0.
\end{equation}
Let $\mathcal{I} \colon V \to V_{h,0}$ be the interpolation operator discussed in Section~\ref{subsec:FEspaces} with the modification in Remark~\ref{rem:IforUandV}.
We define for all $v\in V$ the composed mapping
\begin{align*}
F v \coloneqq \mathcal{I} v - \mathcal{C}_1 \delta - \mathcal{C}_2 ( \delta - \mathcal{C}_1 \delta )\qquad\text{with }\delta \coloneqq v - \mathcal{I} v.
\end{align*}

\textit{Step 2 (Annihilation property of $F$).}
Let $u_h \in U_h$ and $v \in V$.
We set $\delta \coloneqq v - \mathcal{I} v$. An element-wise integration by parts reveals
\begin{align*}
&b(u_h,v-Fv) = b(u_h,\delta -\mathcal{C}_1 \delta - \mathcal{C}_2 ( \delta - \mathcal{C}_1 \delta ))\\
&\quad = \sum_{K\in\tria} \langle \partial_t u_h-\Delta_x u_h ,(1- \mathcal{C}_2)(\delta - \mathcal{C}_1 \delta)\rangle_K + \int_{K_t} \int_{\partial K_x} \nabla_x u_h \cdot \nu_x\,  (1 - \mathcal{C}_1 ) \delta\dx \ds.
\end{align*}
The volume contributions equal zero due to \eqref{eq:DefC2} and the integral over the faces equals zero due to \eqref{eq:DefC1}.

\textit{Step 3 (Boundedness of $F$).}
Let $\widehat{F}:= \mathcal{I} + \mathcal{C}_1(1-\mathcal{I})$, so that $F:= \widehat{F} + \mathcal{C}_2(1-\widehat{F})$.
We show that for $G \in \{\mathcal{I}, \widehat{F}, F\}$, 
\be \label{eq:G}
\|G v\|_V^2+ \sum_{K \in \tria} h_{K,x}^{-2}\, \|(1-G)v\|^2_{L^2(K)} \lesssim \|v\|_V^2.
\ee
This yields in particular $\|F u\|_V \lesssim \|u\|_V$.

For $G=\mathcal{I}$, Corollary~\ref{cor:Stability} and Remark~\ref{rem:ExtByZero} yield the claim in \eqref{eq:G}.
The parabolic scaling assumption in \eqref{eq:ParabolicScaling}, the inverse inequalities in Lemma~\ref{lem:localizedGeneral} and \ref{lem:inverseEst}, and the trace inequality $\|v\|^2_{L^2(f)} \lesssim h_{K,x}^{-1}\|v\|_{L^2(K)}^2+h_{K,x}\|v\|_{L^2(K_t;H^1(K_x))}^2$ for $f \in \mathcal{F}_t$ with $f \subset K \in \tria$ yield
\begin{equation} \label{eq:temp}
\begin{split}
\|  \mathcal{C}_1 v\|_V^2 \lesssim \sum_{K \in \tria} h_{K,x}^{-2}\, \|\mathcal{C}_1 v\|^2_{L^2(K)} &\lesssim \|v\|^2_{L^2(\mathcal{J};H^1_0(\Omega))}+\sum_{K \in \tria} h_{K,x}^{-2}\, \| v\|^2_{L^2(K)} \\
& \lesssim \|v\|^2_{V}+\sum_{K \in \tria} h_{K,x}^{-2}\, \| v\|^2_{L^2(K)}.
\end{split}
\end{equation}
Hence, the definition of $\widehat{F}$ and the already stated result for $\mathcal{I}$ verify \eqref{eq:G} for $G=\widehat{F}$. 
Similarly, we conclude \eqref{eq:G} for $G=F$ from
\begin{equation*}
\|  \mathcal{C}_2 u\|_V^2 \lesssim \sum_{K \in \tria} h_{K,x}^{-2} \|\mathcal{C}_2 u\|^2_{L^2(K)} \lesssim \sum_{K \in \tria} h_{K,x}^{-2} \| u\|^2_{L^2(K)}.
\end{equation*}
This shows that the linear operator $F\colon V \to V_h$ is continuous.
\end{proof}
The existence of the Fortin operator ensures the quasi-optimality of the numerical scheme discussed in the following section. The price to be paid is the parabolic scaling, which is more natural for the parabolic problem but leads to  more degrees of freedom than needed for smooth problems. We counter this effect by using different polynomial degrees in space and time as discussed in Remark~\ref{rem:1}.
\section{Preconditioner}\label{sec:preCond}
As explained in Section~\ref{subsec:discreteSetting}, we have to replace the inner product $\langle \bigcdot , \bigcdot \rangle_V$ in \eqref{eq:SaddlePointProb} by a suitable inner product $(G_h^{-1}\bigcdot)(\bigcdot)$ on $V_h\times V_h$ with induced norm being uniformly equivalent to the norm $\lVert \bigcdot \rVert_V$. The operator $G_h\in \cL(V_h', V_h)$ is a preconditioner that we will introduce in the following two subsections. A key tool in our design and analysis of $G_h$ is the following classical result, cf.~\cite[Sec.~4.1]{Oswald94}.
\begin{theorem}[Additive subspace correction] \label{thm:1}
Let $\U$ be a Hilbert space, and let $(\U_i)_i$ be a family of Hilbert spaces with embeddings $ E_i\colon \U_i \rightarrow \U$ such that
\begin{equation*}
\bigcup_i E_i(\U_i)=\U.
\end{equation*}
We set
\begin{equation*}
\lnrm u\rnrm_\U^2 \coloneqq \inf_{\{(u_i)_{i} \in \prod_{i} \U_i\colon \sum_{i} E_i u_i=u\}} \textstyle{\sum_{i}} \|u_i\|^2_{\U_i}\qquad\text{for all }u\in \U.
\end{equation*}
Let $R_i\colon \U_i' \rightarrow \U_i$ be the Riesz isomorphism. Then
$G \coloneqq \sum_{i} E_i R_i E_i'\in \Lis(\U',\U)$ is a linear isomorphism with
\begin{equation*}
\inf_{0 \neq u \in \U} \frac{\nrm u\nrm_\U}{\|u\|_\U} \leq
\frac{(G^{-1}\bigcdot)(\bigcdot)^{1/2}}{\|\bigcdot\|_{\U}} \leq
\sup_{0 \neq u \in \U} \frac{\nrm u\nrm_\U}{\|u\|_\U} \qquad\text{on } \U\setminus \{0\}.
\end{equation*}
\end{theorem}
%
\subsection{Splitting and bubbles}\label{subsec:Splitting}
Due to the addition of edge bubbles for the construction of our Fortin interpolator, the test spaces $V_h \coloneqq V_{h,0} \oplus V_{h,1} \oplus V_{h,2}$ defined in \eqref{eq:testSpace} are not nested under mesh refinement. Therefore, we decompose the test space into the space without bubbles $V_{h,0}$ and the bubble part $\mathcal{B}_h\coloneqq  V_{h,1} \oplus V_{h,2}$. This splitting will be stable with respect to the $\|\bigcdot\|_V$-norm, allowing us to precondition both parts separately. 
As we will see, on the bubble-part a diagonal preconditioner will suffice. 
The following proposition summarizes these ideas.

\begin{proposition}[Splitting]\label{prop:splitting} Let $V_h=V_{h,0} \oplus \mathcal{B}_h$ such that
\be \label{eq:7}
\|v+b\|_V^2 \eqsim \|v\|_V^2+ \|b\|_V^2 \qquad\text{for all }v \in V_{h,0},\,b \in \mathcal{B}_h.
\ee
Let $G_{h,0}\colon V_{h,0} ' \rightarrow V_{h,0}$ be such that
\be \label{eq:8}
\|\bigcdot\|_V^2 \eqsim (G_{h,0}^{-1}\,\bigcdot)(\bigcdot) \quad \text{on } V_{h,0}.
\ee
Let $\Psi_h$ be a basis for $\mathcal{B}_h$ such that one has for all scalars $c_\psi$ the equivalence
\be \label{eq:9}
\Big\|\sum_{\psi \in \Psi_h} c_\psi \psi\Big\|_V^2 \eqsim \sum_{\psi \in \Psi_h} c_\psi^2.
\ee
Then with $G_{h,1}\colon \mathcal{B}_h'\rightarrow \mathcal{B}_h$ defined by $G_{h,1}f :=\sum_{\psi \in \Psi_h} f(\psi)\psi$, and 
with the inclusions $E_{h,0}\colon V_{h,0} \rightarrow V_h$ and $E_{h,1}\colon \mathcal{B}_h \rightarrow V_h$ the operator
\begin{equation}\label{eq:SplittedPrecond}
G_h\coloneqq E_{h,0} G_{h,0} E_{h,0}'+ E_{h,1} G_{h,1} E_{h,1}'\colon V_h' \rightarrow V_h
\end{equation}
satisfies 
\begin{equation*}
\|\bigcdot\|_V^2 \eqsim (G_h^{-1}\bigcdot)(\bigcdot) \quad \text{on } V_h.
\end{equation*}
\end{proposition}

\begin{proof} For $V_{h,0}$ equipped with scalar product $(G_{h,0}^{-1} \,\bigcdot)(\bigcdot)$, the Riesz isomorphism $V_{h,0}'\rightarrow V_{h,0}$ is $G_{h,0}$.
For $\mathcal{B}_h$ equipped with scalar product 
\begin{equation*}
\Big(\sum_{\psi \in \Psi} c_\psi \psi,\sum_{\phi \in \Psi} d_\phi \phi\Big) \mapsto 
\sum_{\psi \in \Psi} c_\psi d_\phi,
\end{equation*}
the Riesz isomorphism $\mathcal{B}_h'\rightarrow \mathcal{B}_h$ is $G_{h,1}$. Now the result follows from Theorem~\ref{thm:1}, where the (uniform) equivalence of the triple-bar norm with $\|\bigcdot\|_V$ on $V_h$ follows from the assumptions in \eqref{eq:7}--\eqref{eq:9}.
\end{proof}
The following lemma verifies the assumptions of Proposition~\ref{prop:splitting}.
\begin{lemma}[Verification of assumptions]
The spaces $V_{h,0}$ and $\mathcal{B}_h \coloneqq V_{h,1} \oplus V_{h,2}$ defined in \eqref{eq:testSpace} satisfy the assumptions in \eqref{eq:7} and \eqref{eq:9} with $\Psi_h$ consisting of scaled Lagrange basis functions.
\end{lemma}
\begin{proof}

Let $(\psi_j)_{j=1}^N$ denote the nodal basis of $V_{h,0}$ and let $(\psi_j)_{j=N+1}^M $ denote a local basis of $\mathcal{B}_h$ that is $L^2(Q)$-normalized, i.e.~$\lVert \psi_j \rVert_{L^2(Q)}=1$ for all $j=N+1,\dots,M$.
We define the Scott--Zhang interpolation operator $\overline{\mathcal{I}}\colon W \to V_h$ as in Definition~\ref{def:SZ}, that is, with suitable bi-orthogonal weights $(\psi_j^*)_{j=1}^{N+M}$ the operator reads 
\begin{equation*}
\overline{\mathcal{I}} \coloneqq \sum_{j=1}^{N+M} \langle \psi_j^*,\bigcdot \rangle_{S_j} \psi_j.
\end{equation*}
Moreover, we set its reduced version $\mathcal{I}\colon W \to V_{h,0}$ as
\begin{equation*}
\mathcal{I} \coloneqq \sum_{j=1}^N \langle \psi_j^*,\bigcdot \rangle_{S_j} \psi_j.
\end{equation*}
Let $v\in V_{h,0}$ and $b \in \mathcal{B}_h$.
By definition we obtain $\mathcal{I} b = 0$. 
This identity and the stability properties of $\mathcal{I}$ stated in Corollary~\ref{cor:Stability} yield
\begin{equation*}
\lVert v \rVert_V + \lVert b \rVert_V = \lVert \mathcal{I} (v + b) \rVert_V + \lVert (1-\mathcal{I}) (v + b) \rVert_V  \lesssim \lVert v+b\rVert_V.
\end{equation*}
This verifies \eqref{eq:7}.

Theorem~\ref{thm:LocalizationIIold}, Lemma~\ref{lem:inverseEst}, and Lemma~\ref{lem:locApx} yield the norm equivalence
\begin{equation}\label{eq:NormEquiB}
\begin{aligned}
\|b\|_V^2 & = \|(1-\mathcal{I}) b\|_V^2  \lesssim \sum_{K \in \tria} h_{K,t}^{-1}\, \|b\|_{L^2(K)}^2 = \sum_{K \in \tria} h_{K,t}^{-1}\, \| (1-\mathcal{I})b\|_{L^2(K)}^2 \\
& \lesssim \|b\|_V^2.
\end{aligned}
\end{equation}
Since the number $\#\{ j \in \lbrace N+1,\dots,M\rbrace\colon |\supp \psi_j \cap K|>0\} \eqsim 1$ of bubble basis functions supported on each $K\in \tria$ is uniformly bounded, we obtain for any $b\in \mathcal{B}_h$ with representation $b = \sum_{j=N+1}^M c_j \psi_j$ and $K\in \tria$ the equivalence
\begin{equation*}
\Big\|\sum_{\substack{j = N+1 ,\dots,M \\ |\supp \psi_j \cap K|>0}} c_j \psi_j\Big\|_{L^2(K)} ^2\eqsim \sum_{\substack{j = N+1 ,\dots,M \\ |\supp \psi_j \cap K|>0}} c_j^2\, \| \psi_j\|_{L^2(K)}^2.
\end{equation*}
This bound, \eqref{eq:NormEquiB}, and the mesh properties stated in Assumption~\ref{ass:Partition}  verify \eqref{eq:9} with basis 
\begin{equation*}
\Psi_h = \big(\textup{diam}_x(\supp(\psi_j))  \, \psi_j\big)_{j=N+1}^M.\qedhere
\end{equation*}
\end{proof}

\subsection{Additive subspace correction}\label{subsec:AddSubspaceCorr}
In this section we design a preconditioner $G_{h,0}= G_L\colon V_L'\to V_L$ with $V_L \coloneqq W_{h,0}^p \subset V$ with the property \eqref{eq:8}. 
In view of a later construction of a preconditioner at the trial side, we consider the more general situation of arbitrary polynomial degrees $p = (p_t,p_x)\in \mathbb{N}^2$. 
In order to design $G_L$, we generalize the fractional spaces defined in \eqref{eq:InterpolSpaceTime}, that is, we set for any $s \in [0,1]$ the interpolation spaces
\begin{align*}
&\cH^{2s}(\Omega)\coloneqq [H^1_0(\Omega) \cap H^2(\Omega),L^2(\Omega)]_{1-s,2},\quad
\cH^s(\cJ)\coloneqq [H^1_{,0}(\cJ),L^2(\cJ)]_{1-s,2},\\
&\cH^{2s,s}(Q)\coloneqq [L^2(\cJ;\cH^2(\Omega)) \cap \cH^1(\cJ;L^2(\Omega)),L^2(Q)]_{1-s,2}.
\end{align*}
In particular, we obtain the isomorphisms
\begin{align*}
\cH^{1,1/2}(Q)& \simeq L^2(\cJ;\cH^{1}(\Omega)) \cap \cH^{1/2}(\cJ;L^2(\Omega))\\
& \simeq L^2(\cJ;H^1_0(\Omega)) \cap H^{1/2}_{,0}(\cJ;L^2(\Omega)) = V.
\end{align*}

The operator $G_L$ will be of multi-level type. We start our design with collecting some ingredients for the \emph{uniform refinement case}.
Recall our initial partition $\tria_0 = \tria_t \otimes \tria_x$ of $Q$ and refinement strategy discussed in Section~\ref{subsec:tria}. Uniform refinements lead to a nested sequence of conforming partitions 
\begin{equation*}
\tria_0 = \widehat{\tria}_0 < \widehat{\tria}_1 < \widehat{\tria}_2 < \cdots .
\end{equation*}
Each prism $K=K_t \times K_x \in \widehat{\tria}_\ell$ with $\ell\in \mathbb{N}_0$ consists of a time interval $K_t$ with $\diam(K_t) = |K_t| \eqsim \varrho^{-2\ell}$ and a uniformly shape regular $d$-simplex $K_x$ with $\diam( K_x)\eqsim \varrho^{-\ell}$ with constant $\varrho = 2$. 
For each uniform refinement $\widehat{\tria}_\ell$, $\ell \in \mathbb{N}_0$, the corresponding finite element spaces read
\begin{equation*}
\widehat{V}_\ell \coloneqq \{v \in \cH^{1,1/2}(Q)\colon v|_{K_t \times K_x} \in \P_{p_t}(K_t) \otimes \P_{p_x}(K_x) \text{ for all }K \in \widehat{\tria}_\ell\} \subset V.
\end{equation*}
Classical approximation results verify the Jackson estimate
\begin{equation}\label{eq:Jackson}
\min_{w \in \widehat{V}_\ell}\|u-w\|_{L^2(Q)} \lesssim \varrho^{-2\ell} \|u\|_{\cH^{2,1}(Q)}\qquad \text{for all }u \in \cH^{2,1}(Q).
\end{equation}
Writing $\widehat{V}_\ell = \widehat{V}_{\ell,t} \otimes \widehat{V}_{\ell,x}$, for $s \in \{0,1\}$, and thus by interpolation for all $s \in [0,1]$, we have $\|\bigcdot\|_{\cH^{s}(\cJ)} \lesssim \varrho^{2\ell s}\|\bigcdot\|_{L^2(\cJ)}$ on $\widehat{V}_{\ell,t}  $, and so $\|\bigcdot\|_{\cH^{s}(\cJ;L^2(\Omega))} \lesssim \varrho^{2\ell s}\|\bigcdot\|_{L^2(Q)}$ on $\widehat{V}_{\ell} $. For $s<3/4$, we have $\|\bigcdot\|_{\cH^{2s}(\Omega)} \lesssim \varrho^{2\ell s}\|\bigcdot\|_{L^2(\Omega)}$ on $\widehat{V}_{\ell,x}$ \cite{75.64}, and so 
$\|\bigcdot\|_{L^2(\cJ;\cH^{2s}(\Omega))} \lesssim \varrho^{2\ell s}\|\bigcdot\|_{L^2(Q)}$ on $\widehat{V}_{\ell}$. Since $\cH^{2s,s}(Q) \simeq L^2(\cJ;\cH^{2s}(\Omega)) \cap \cH^{s}(\cJ;L^2(\Omega))$, these bounds yield for any $\gamma \in [0,3/4)$ the Bernstein estimate
\begin{equation}\label{eq:Bernstein}
 \|w\|_{\cH^{2\gamma,\gamma}(Q)} \lesssim \varrho^{2 \ell \gamma} \|w\|_{L^2(Q)} \qquad \text{for all }w \in \widehat{V}_\ell. 
\end{equation}

As a consequence of \eqref{eq:Jackson}--\eqref{eq:Bernstein}, and the nesting $\widehat{V}_\ell \subset \widehat{V}_{\ell+1}$, we have the following result involving the $L^2(Q)$-orthogonal projector onto $\widehat{V}_i$ denoted by
\begin{equation*}
P_\ell \colon L^2(Q) \to \widehat{V}_\ell\qquad\text{with } P_{-1}\coloneqq 0.
\end{equation*}
\begin{theorem}[Norm equivalence I]\label{thm:NormEqui}
 One has for $s \in (-\frac34,\frac34)$ the equivalence
\begin{equation*}
\|u\|^2_{\cH^{2s,s}(Q)} \eqsim \sum_{\ell=0}^\infty \varrho^{4s\ell}\|(P_\ell -P_{\ell-1})u\|_{L^2(Q)}^2 \qquad \text{for all }u \in \cH^{2s,s}(Q),
\end{equation*}
where $\cH^{2s,s}(Q)\coloneqq \cH^{-2s,-s}(Q)'$ for $s<0$.
\end{theorem}
\begin{proof}
This result is shown in \cite[Thm.~2.1]{DahmenStevenson99}.
\end{proof}
\begin{corollary}[Norm equivalence II] \label{corol:1} One has for all $s \in (0,\frac34)$ and $u \in \cH^{2s,s}(Q)$
\begin{equation*}
\sum_{\ell=0}^\infty \varrho^{4s\ell}\|(\identity - P_\ell)u\|_{L^2(Q)}^2 + \lVert P_0 u \rVert_{L^2(Q)}^2 \eqsim \|u\|^2_{\cH^{2s,s}(Q)}.
\end{equation*}
\end{corollary}
\begin{proof}
Any $u\in \cH^{2s,s}(Q)$ satisfies due to Theorem~\ref{thm:NormEqui}
\begin{align*}
&\sum_{\ell=0}^\infty \varrho^{4s\ell}\|(\identity - P_\ell)u\|_{L^2(Q)}^2
=\sum_{\ell=0}^\infty \varrho^{4s\ell}\sum_{j=0}^\infty \|(P_j-P_{j-1})(\identity - P_\ell)u\|_{L^2(Q)}^2
\\
&=\sum_{\ell=0}^\infty \varrho^{4s\ell} \sum_{j > \ell} \|(P_{j+1} - P_{j})u\|_{L^2(Q)}^2
\\
&=\sum_{j=1}^\infty \Big(\sum_{\ell=0}^{j-1} \varrho^{4s(\ell-j-1)} \Big)\varrho^{4s(j+1)}  \|(P_{j+1} - P_j)u\|_{L^2(Q)}^2\\
&\leq \tfrac{1}{1-\varrho^{-4 s}}\sum_{j=1}^\infty \varrho^{4s(j+1)}  \|(P_{j+1} - P_j)u\|_{L^2(Q)}^2\lesssim  \|u\|^2_{\cH^{2s,s}(Q)}.
\end{align*}
Combining this estimate with  $\lVert P_0 u \rVert_{L^2(Q)}^2\leq \lVert u \rVert_{L^2(Q)}^2 \lesssim \|u\|^2_{\cH^{2s,s}(Q)}$ yields the lower bound.
The upper bound follows from
\begin{align*}
\|u\|^2_{\cH^{2s,s}(Q)} &\eqsim \sum_{\ell=0}^\infty \varrho^{4s\ell}\|(P_\ell-P_{\ell-1})u\|_{L^2(Q)}^2 \\
&= \sum_{\ell=1}^\infty \varrho^{4s\ell}\|P_\ell( \identity -P_{\ell-1})u\|_{L^2(Q)}^2 + \lVert P_0 u\rVert_{L^2(Q)}^2\\
 &\leq \varrho^{4s} \sum_{\ell=0}^\infty \varrho^{4s\ell}\|(\identity - P_\ell)u\|_{L^2(Q)}^2 + \lVert P_0 u\rVert_{L^2(Q)}^2. \qedhere
\end{align*}
\end{proof}
\begin{proposition}[Strengthened Cauchy-Schwarz] \label{prop:1} 
For $s \in (0,3/4)$ one has
\begin{equation*}
\Big\|\sum_{\ell=0}^\infty v_\ell\Big\|_{\cH^{2s,s}(Q)}^2 \lesssim \sum_{\ell=0}^\infty \varrho^{4 s \ell} \|v_\ell\|_{L^2(Q)}^2 \qquad \text{for all }(v_\ell)_{\ell=0}^\infty \in (\widehat{V}_\ell)_{\ell=0}^\infty .
\end{equation*}
\end{proposition}

\begin{proof} Let $\eps>0$ be such that $s \pm \eps \in (0,3/4)$ and let $(v_\ell)_{\ell=1}^\infty \in (\widehat{V}_\ell)_{\ell=1}^\infty$. By \eqref{eq:Bernstein} we obtain for any indices $\ell,j\in \mathbb{N}_0$
\begin{align*}
|\langle v_\ell,v_j\rangle_{\cH^{2s,s}(Q)}| &\leq \|v_\ell\|_{\cH^{2(s-\eps),s-\eps}(Q)}\|v_j\|_{\cH^{2(s+\eps),s+\eps}(Q)}\\
& \lesssim \varrho^{2\eps(j-\ell)}  \varrho^{2 s \ell} \|v_\ell\|_{L^2(Q)} \varrho^{2 s j} \|v_j\|_{L^2(Q)}.
\end{align*}
Hence, thanks to the estimate $\sum_{\ell= 0}^\infty a_\ell \big(\sum_{j=0}^\ell \rho^{-(\ell-j)} a_j\big)\leq (1-\rho^{-1})^{-1} \sum_{\ell=0}^\infty a_\ell^2$ for $(a_\ell)_{\ell=0}^\infty \subset \mathbb{R}$ and $\rho >1$, and the Cauchy-Schwarz inequality in $\ell^2(\mathbb{N}_0)$, we obtain
\begin{align*}
&\Big\|\sum_{\ell=0}^\infty v_\ell\Big\|_{\cH^{2s,s}(Q)}^2  = \sum_{\ell,j=0}^\infty \langle v_\ell,v_j\rangle_{\cH^{2s,s}(Q)} \lesssim \sum_{\ell,j=0}^\infty  \varrho^{-2\eps|j-\ell|}  \varrho^{2 s \ell} \|v_\ell\|_{L^2(Q)} \varrho^{2 s j} \|v_j\|_{L^2(Q)} \\
&\qquad = \sum_{\ell=0}^\infty \varrho^{4s\ell} \|v_\ell\|_{L^2(Q)}^2 + 2\, \sum_{\delta = 1}^\infty \varrho^{-2\eps \delta} \sum_{j=0}^\infty \varrho^{2 s (j+\delta)} \|v_{j+\delta}\|_{L^2(Q)} \varrho^{2 s j} \|v_{j}\|_{L^2(Q)}  \\
&\qquad\ \leq \sum_{\ell=0}^\infty \varrho^{4s\ell} \|v_\ell\|_{L^2(Q)}^2 + 2\, \sum_{\delta = 1}^\infty \varrho^{-2\eps \delta}  \sum_{j=0}^\infty \varrho^{4 s j} \|v_{j}\|^2_{L^2(Q)}\\
& \qquad = \frac{1 + \varrho^{-2\eps}}{1-\varrho^{-2\eps}} \sum_{\ell=0}^\infty \varrho^{4s\ell} \|v_\ell\|_{L^2(Q)}^2.\qedhere
\end{align*}
\end{proof}
By applying the \emph{adaptive refinement strategy} described in Section~\ref{subsec:tria} we obtain a nested sequence of generally non-conforming partitions $\tria_i$ of $Q$, that is
\begin{equation*}
\tria_0 < \tria_1 < \tria_2 < \cdots.
\end{equation*}
The elements of $\tria_i$ are prisms $K=K_t \times K_x$, where $K_t$ is an interval,
$K_x$ is a uniformly shape regular $d$-simplex, and
\begin{equation*}
\diam (K_t) \eqsim \diam (K_x)^2\qquad\text{so that}\qquad \diam (K) \eqsim \diam (K_x).
\end{equation*}
According to Assumption~\ref{ass:Partition} the partitions are uniformly \emph{graded} in the sense that neighboring prisms in $\tria_i$ have uniformly comparable diameters.
Set for all $i\in \mathbb{N}_0$ the discrete spaces 
\begin{equation*}
V_i \coloneqq V(\tria_i) \coloneqq \{v \in V \colon v|_K \in \P_{p_t}(K_t) \otimes \P_{p_x}(K_x) \text{ for all }K = K_t \times K_x \in \tria_i\}\subset V.
\end{equation*}
We equip the space $V_i$ with the usual \emph{nodal basis} denoted by $(\psi_{i,j})_{j \in \mathcal{N}^p(\tria_i)}$, where $I_i:=\mathcal{N}^p(\tria_i)$. We denote the support and spatial diameter of these functions by
\begin{equation*}
Q_{i,j}\coloneqq \supp (\psi_{i,j})\qquad \text{and} \qquad h_{i,j}\coloneqq \diam_x (Q_{i,j}).
\end{equation*}
Let $\mathcal{I}_i$ be the corresponding Scott--Zhang projector onto $V_i$ from Definition~\ref{def:SZ}.
These projectors are uniformly \emph{local} in the sense that for $j \in I_i$, for some
connected domain $\tilde{Q}_{i,j}\supset Q_{i,j}$ that is the union of a uniformly bounded number of prisms $K \in \tria_{i-1}$, it holds that 
\be \label{eq:6mod}
\begin{aligned}
(\mathcal{I}_{i-1} v)|_{Q_{i,j}} &\text{ depends only on }v|_{\tilde{Q}_{i,j}} \text{ and } \tria_{i-1}|_{\tilde{Q}_{i,j}},\\
(\mathcal{I}_{i} v)|_{Q_{i,j}} &\text{ depends only on }v|_{\tilde{Q}_{i,j}} \text{ and } \tria_{i}|_{\tilde{Q}_{i,j}}.
%
\end{aligned}
\ee
As a consequence, in the case that $\tria_i|_{\tilde{Q}_{i,j}} =\tria_{i-1}|_{\tilde{Q}_{i,j}}$, it holds that 
 $(\mathcal{I}_i v)|_{Q_{i,j}}=(\mathcal{I}_{i-1} v)|_{Q_{i,j}}$. In other words, $(\mathcal{I}_i-\mathcal{I}_{i-1})v$ vanishes outside the direct vicinity of the union of the prisms $K \in \tria_i \setminus \tria_{i-1}$. This implies that for each index $i\in \mathbb{N}$, there exists a set of nodes $J_i \subset I_i$ with
 \begin{equation}\label{eq:PropI}
 \ran(\mathcal{I}_i-\mathcal{I}_{i-1}) \subset \Span\{\psi_{i,j} \colon j \in J_i\} \quad \text{and}\quad \# J_i \lesssim \#(\tria_i \setminus \tria_{i-1}).
\end{equation}
With $J_0 \coloneqq I_0$, we have $\ran \mathcal{I}_0 = \Span\{\psi_{0,j} \colon j \in J_0\}$, and $\# J_0 \eqsim \# \tria_0$.
Since we have for any $L \in \N$ that $\dim V_L \eqsim \# \tria_L$, and \cite[eq. (3.1)]{BinevDeVore04}
\begin{equation*}
\# \tria_0+ \sum_{i=1}^L \# (\tria_i \setminus \tria_{i-1}) \leq 2 \# \tria_L,
\end{equation*}
it holds that
\be \label{eq:compl}
\sum_{i=0}^L \# J_j \lesssim \dim V_L.
\ee

The following property, that locally links the adaptively refined meshes $(\tria_i)_i$ to the uniform meshes $(\widehat{\tria}_\ell)_\ell$, holds true thanks to the uniform grading of the meshes $\tria_i$: There exists a constant $q \in \N$ such that for all $i \in \N_0$ and $j \in I_i$ there exists an index $\ell(i,j) \in \N_0$ with
\be \label{eq:1}
\widehat{\tria}_{\ell(i,j)}|_{\tilde{Q}_{i,j}} \leq \tria_{i-1}|_{\tilde{Q}_{i,j}} \leq \tria_i|_{\tilde{Q}_{i,j}} \leq \widehat{\tria}_{\ell(i,j)+q}|_{\tilde{Q}_{i,j}}.
\ee
The proof of the following theorem is based on arguments from \cite{WuZheng17} in which a (multilevel) multiplicative subspace correction method is analyzed where the underlying space is $H^1_0(\Omega)$, instead of $\cH^{2s,s}(Q)$. The strengthened Cauchy-Schwarz inequality that applies in that setting is substituted by Proposition~\ref{prop:1}.

\begin{theorem}[Norm equivalence]\label{thm:normEqui} For $s \in (0,3/4)$, $L \in \N_0$, and $v \in V_L$ one has with $J_i$ defined for all $i\in \mathbb{N}$ in \eqref{eq:PropI}
\begin{equation*}
\|v\|^2_{\cH^{2s,s}(Q)} \eqsim \nrm v\nrm_{L,s}^2 \coloneqq
\inf_{\{(v_{i,j})\colon v_{i,j}\in \Span\{\psi_{i,j}\},\,v=\sum\limits_{i=0}^L\sum\limits_{j \in J_i} v_{i,j}\}} 
\sum_{i=0}^L\sum_{j \in J_i} h_{i,j}^{-4s}\|v_{i,j}\|_{L^2(Q)}^2.
\end{equation*}
\end{theorem}

\begin{proof} Let $v \in V_L$. With $v_{i,j}\in \Span\{\psi_{i,j}\}$ such that $ (\mathcal{I}_i-\mathcal{I}_{i-1})v=\sum_{j \in J_i} v_{i,j}$, we decompose $v=\sum_{i=0}^L  (\mathcal{I}_i-\mathcal{I}_{i-1})v=\sum_{i=0}^L \sum_{j \in J_i}v_{i,j}$.
We claim that this decomposition is (uniformly) locally $L^2$-stable in the sense that
\begin{equation}\label{eq:Claim}
\sum_{j \in J_i} \|v_{i,j}\|_{L^2(K)}^2 \eqsim \Big\| \sum_{j \in J_i} v_{i,j}\Big\|_{L^2(K)}^2\qquad \text{for all }K\in \tria_i.
\end{equation}
Indeed, the inequality $\gtrsim$ follows from the local supports of the nodal basis functions, and so we focus on the reversed inequality.
Recall the set of Lagrange nodes $I_i\coloneqq\mathcal{N}^p(\tria_i)$ in $V_i$.
Since, modulo a harmless scaling, given an initial mesh $\tria_0$ there are only finitely many local mesh configurations, it suffices to show that with coefficients $(c_j)_{j \in I_i}\subset \mathbb{R}$
\be \label{eq:localindependence}
\sum_{j \in I_i} c_j \psi_{i,j}|_K=0\qquad \text{implies}\qquad c_j \psi_{i,j} |_{\text{int}(K)}=0  \text{ for all }j \in I_i.
\ee
The index set $I_i$ is the union over $K \in \tria_i$ of those local Lagrange nodes on $K$ that are `free', i.e.~not on a (master) facet $\tilde{F}$ of $\tilde{K} \in \tria_i$ that includes a (slave) facet $F$ of $K$ as a proper subset. The condition \eqref{eq:LevelDist} ensures that all local Lagrange nodes on $\tilde{K}$ that are on such a master facet $\tilde{F}$ are free
(cf.~\cite{GantnerStevenson23}).

In the case that all nodes of $K$ are free, then the only $\psi_{i,j}$ with $\psi_{i,j}|_{\text{int}(K)} \neq 0$ correspond to nodes in $K$, and \eqref{eq:localindependence} is obvious.
Otherwise, when $K$ has nodes that are not free, it has one or more slave facets.
If $u\coloneqq \sum_{j \in I_{i,j}} c_i \psi_{i,j}|_K=0$, then for any of those slave facets $F$, the function $u|_F$ vanishes, which means that for the corresponding neighbouring $\tilde{K}$ with master facet $\tilde{F}$, also $u|_{\tilde{F}}$ vanishes so that its degrees of freedom on $\tilde{F}$ vanish. Since the only possible $j \in I_i$ with $j \not\in K$ and $\psi_{i,j}|_{\text{int}(K)} \neq 0$ correspond to degrees of freedom of this type, the proof of \eqref{eq:localindependence} is completed. This verifies the claim in \eqref{eq:Claim}.

By the (uniform) grading of $\tria_i$, the locality of $v_{i,j}$, and \eqref{eq:Claim}, we infer that
\begin{align*}
\sum_{j \in J_i} h_{i,j}^{-4s} \| v_{i,j}\|_{L^2(Q_{i,j})}^2 \eqsim 
\sum_{j \in J_i} \sum_{K \in \tria_i} h_K ^{-4s} \| v_{i,j}\|_{L^2(K)}^2 =
\sum_{K \in \tria_i} h_K ^{-4s} \sum_{j \in J_i}  \| v_{i,j}\|_{L^2(K)}^2\\
 \eqsim
\sum_{K \in \tria_i} h_K ^{-4s} \|(\mathcal{I}_i-\mathcal{I}_{i-1})v\|_{L_2(K)}^2 \eqsim
\sum_{j \in J_i} h_{i,j} ^{-4s} \|(\mathcal{I}_i-\mathcal{I}_{i-1})v\|_{L_2(Q_{i,j})}^2.
\end{align*}

Property \eqref{eq:6mod} shows that $((\mathcal{I}_i-\mathcal{I}_{i-1})v)|_{Q_{i,j}}$ only depends on $v|_{\tilde{Q}_{i,j}}$.
This result in combination with \eqref{eq:1}, and the fact that ${\mathcal I}_i$ and ${\mathcal I}_{i-1}$ are projectors onto $V_i$ and $V_{i-1}$, respectively, show that  $((\mathcal{I}_i-\mathcal{I}_{i-1})v)|_{Q_{i,j}}=0$ for $v \in \widehat{V}_{\ell(i,j)}$.
We infer that
\begin{align} \label{eq:2}
\begin{aligned}
&\sum_{i=0}^L \sum_{j \in J_i} h_{i,j}^{-4s}\|v_{i,j}\|^2_{L^2(Q)} \eqsim \sum_{i=0}^L \sum_{j \in J_i}  h_{i,j}^{-4s} \|(\mathcal{I}_i-\mathcal{I}_{i-1})v\|_{L^2(Q_{i,j})}^2\\  
&\qquad \eqsim \sum_{i=0}^L \sum_{j \in J_i}  h_{i,j}^{-4s} \|\mathcal{I}_i(\identity-P_{\ell(i,j)})v-\mathcal{I}_{i-1} (\identity-P_{\ell(i,j)})v\|_{L^2(Q_{i,j})}^2\\ 
&\qquad \lesssim \sum_{i=0}^L \sum_{j \in J_i}  \varrho^{4s \ell(i,j)} \|(\identity-P_{\ell(i,j)})v\|_{L^2(\tilde{Q}_{i,j})}^2.
\end{aligned}
\end{align}
where for the last inequality we used the uniform $L_2(Q)$-boundedness and the locality of ${\mathcal I}_i$ and ${\mathcal I}_{i-1}$, as well as $\varrho^{-\ell(i,j)} \eqsim h_{i,j}$ by \eqref{eq:1}.

Since there exists only a finite number of overlapping patches $\tilde{Q}_{i,j}$ with the same associated value $\ell(i,j)$, one has
\begin{equation*}
\sum_{\{0 \leq i \leq L,\,j\in J_i\colon \ell(i,j)=\ell\}} \|\bigcdot\|_{L^2(\tilde{Q}_{i,j})}^2 \lesssim \|\bigcdot\|_{L^2(Q)}^2.
\end{equation*}
Hence, the expression in \eqref{eq:2} can be bounded by an application of Corollary~\ref{corol:1} by some constant multiple of
\begin{equation*}
\sum_{\ell=0}^\infty \varrho^{ 4 s \ell}  \|(\identity-P_\ell)v\|_{L^2(Q)}^2 \lesssim \|v\|^2_{\cH^{2s,s}(Q)}.
\end{equation*}

Conversely, let $v=\sum_{i=0}^L\sum_{j \in J_i} v_{i,j}$ for \emph{some} $v_{i,j}\in \Span\{\psi_{i,j}\}$.
Then $v_{i,j} \in \widehat{V}_{\tilde{\ell}(i,j)}$ with $\tilde{\ell}(i,j) \coloneqq \ell(i,j)+q$. We estimate, using  Proposition~\ref{prop:1},
\begin{align*}
&\sum_{i=0}^L \sum_{j \in J_i} h_{i,j}^{-4s}\|v_{i,j}\|^2_{L^2(Q)}  \eqsim
\sum_{i=0}^L \sum_{j \in J_i} \varrho^{4s\tilde{\ell}(i,j)}\|v_{i,j}\|^2_{L^2(Q)}\\
 &\qquad\qquad= \sum_{\tilde{\ell}=0}^\infty \varrho^{4s\tilde{\ell}} \sum_{\{0 \leq i \leq L,\,j \in J_i\colon\tilde{\ell}(i,j)=\tilde{\ell}\}}\|v_{i,j}\|^2_{L^2(Q)} \\
 &\qquad\qquad\gtrsim \sum_{\tilde{\ell}=0}^\infty \varrho^{4s\tilde{\ell}}\, \Big\|\sum_{\{0 \leq i \leq L,\,j \in J_i\colon\tilde{\ell}(i,j)=\tilde{\ell}\}}v_{i,j}\Big\|^2_{L^2(Q)} \\
 &\qquad\qquad \gtrsim \Big\|\sum_{\tilde{\ell}=0}^\infty \sum_{\{0 \leq i \leq L,\,j \in J_i\colon\tilde{\ell}(i,j)=\tilde{\ell}\}} v_{i,j}\Big\|_{\cH^{2s,s}(Q)}^2=
\|v\|_{\cH^{2s,s}(Q)}^2.\qedhere
\end{align*}
\end{proof}

Combining Theorem~\ref{thm:1} and~\ref{thm:normEqui} leads to the following result.

\begin{corollary}[Preconditioner] 
For $s \in (0,3/4)$, $0 \leq i \leq L$ and $j \in J_\ell$, let the Riesz isomorphism for $\Span \{\psi_{i,j}\}$ equipped with norm $h_{i,j}^{-2s}\|\bigcdot\|_{L^2(Q)}$ be denoted by
\begin{equation*}
R_{i,j}\colon \Span\{\psi_{i,j}\}' \rightarrow \Span \{\psi_{i,j}\},\quad g \mapsto h_{i,j}^{4s} \frac{g(\psi_{i,j})}{\|\psi_{i,j}\|_{L^2(Q)}^2} \psi_{i,j}.
\end{equation*}
Let $E_{i,j}\colon \Span \{\psi_{i,j}\}\to V_L$ denote the inclusion map of $\Span \{\psi_{i,j}\}$ into $V_L$ and set 
\begin{equation*}
G_L \coloneqq \sum_{i=0}^L \sum_{j \in J_i} E_{i,j} R_{i,j} E_{i,j}'\colon V_L' \rightarrow V_L.
\end{equation*}
Then $G_L$ is invertible with
\begin{equation*}
\inf_{0 \neq u \in V_L} \frac{\nrm u\nrm_{L,s}}{\|u\|_{\cH^{2s,s}(Q)}} \leq \frac{(G_L^{-1}\bigcdot)(\bigcdot)^{1/2}}{\|\bigcdot\|_{\cH^{2s,s}(Q)}} 
\leq \sup_{0 \neq u \in V_L} \frac{\nrm u\nrm_{L,s}}{\|u\|_{\cH^{2s,s}(Q)}} 
\quad\text{on } V_L \setminus \{0\}.
\end{equation*}
\end{corollary}

\section{Solving the problem}\label{sec:SolvProb}
Recall the preconditioner $G_h$ defined in the previous section and the operator $B$ defined in \eqref{eq:defB}. We define the operator $A_h\in \Lis(U_h,U_h')$ and the right-hand side $F \in U_h'$ for all  $u_h, w_h \in U_h$ by
\begin{equation*}
\begin{aligned}
(A_h u_h) (w_h) &\coloneqq (Bw_h) (G_h B u_h) + \langle u_h(0),w_h(0)\rangle_\Omega,\\
F_{h}(w_h) &\coloneqq  (Bw_h) (G_h f) + \langle u_0,w_h(0)\rangle_\Omega. 
\end{aligned}
\end{equation*}
The problem in \eqref{eq:system} then reads
\begin{equation}\label{eq:MatrixProb}
A_h u_h = F_h.
\end{equation}  
Thanks to \eqref{eq:compl}, a recursive implementation of $G_h$ that makes use of the nesting of the partitions $\tria_i$ can be performed in linear complexity. Yet, this operator has no sparse representation with respect to common finite element bases.
In order to solve \eqref{eq:MatrixProb} we apply the (Preconditioned) Conjugate Gradient ((P)CG) method.
\subsection{(P)CG-scheme}\label{subsec:PCGscheme}
Knowing that $(A_h w_h)(w_h) \eqsim \langle w_h , w_h \rangle_U$ for $w_h \in U_h$, in order to precondition the CG-scheme, we seek to employ an efficiently applicable preconditioner $K_h\in \Lis(U_h',U_h)$ with the property that $(K_h^{-1} \bigcdot)(\bigcdot)$ is an inner product on $U_h \times U_h$ with norm equivalence
\begin{equation}\label{eq:EquiKh}
\langle w_h , w_h \rangle_U \eqsim (K_h^{-1}w_h)(w_h) \qquad\text{for all }w_h \in U_h. 
\end{equation} 
Indeed, with the spectral condition number $\kappa_h$ of $K_h A_h$, being the quotient of the maximum and minimum of 
\begin{equation*}
\frac{(A_h w_h)(w_h)}{(K_h^{-1} w_h)(w_h)}\quad \text{over}\quad w_h \in U_h,
\end{equation*}
it is known that after $i$ PCG iterations an initial error is reduced in the $(A_h \bigcdot)(\bigcdot)^{1/2}$-norm with at least a factor $(2 \kappa_h^i)/(1+\kappa_h^{2i})$.

When the space $U$ would be equipped with the norm in $W$, cf.~\eqref{eq:defNormW}, the construction in Subsection~\ref{subsec:AddSubspaceCorr} of the optimal multi-level preconditioner on $W_{h,0}^{p_t,p_x} \cap V$, directly extends by removing the homogeneous Dirichlet boundary condition at end time $t=T$.
However, the norm $\norm{\bigcdot}_U$ in $U_h$ is slightly stronger. As we have seen in Remark~\ref{rem:remmie}, we have the continuous embeddings
$L^2(\cJ,H^1_0(\Omega)) \cap H^{1/2+\varsigma}(\cJ;L^2(\Omega)) \hookrightarrow U \hookrightarrow W$  for any fixed $\varsigma>0$.
A standard inverse inequality shows with $h_{t,\min}\coloneqq \min_{K \in \tria} h_{K,t}$ for all $w_h \in U_h$ that 
\begin{equation*}
\|w_h\|_{L^2(\cJ,H^1_0(\Omega)) \cap H^{1/2+\varsigma}(\cJ;L^2(\Omega))} \lesssim h_{t,\min}^{-\varsigma} \|w_h\|_W.
\end{equation*}
Hence, we obtain the estimate
\begin{equation} \label{eq:inequalities}
\|w_h\|_W \lesssim \|w_h\|_U \lesssim h_{t,\min}^{-\varsigma} \|w_h\|_W.
\end{equation}
This estimate leads to the upper bound
\begin{equation*}
\kappa_h \lesssim h_{t,\min}^{-2\varsigma}.
\end{equation*}
The performance of this effective, though suboptimal, preconditioner is examined in our numerical experiment displayed in Section~\ref{subsec:ExpPreCond} below.
\begin{remark}[Homogeneous initial data]
If $u_0 =0$, the variational formulation outlined in Section~\ref{subsec:FracSetting} can be employed using the trial space \( U_F \) and the test space \( V=V_F \). 
The construction of the optimal preconditioner for $V_h \subset V_F$ from Sect.~\ref{sec:preCond} equally well applies to $U_h \subset U_F$ by moving the homogeneous boundary condition at end time $t=T$, which is incorporated in the definition of $V_F$, to the initial time $t=0$, and thus yields an optimal preconditioner for $U_h$.

Since initial data $u_0 \in H^1_0(\Omega)$ has an easy (bounded) extension to a function in $U$, the heat problem for such an initial condition can be reduced to a heat problem with a homogeneous initial condition. 
\end{remark}
%

\subsection{A posteriori error control}\label{subsec:Aposteriori}
A notable advantage of minimal residual methods is their built-in error control. Specifically, the error is equivalent to the computable residual, supplemented by a data approximation error, see for example \cite[Sec.~3.4]{MonsuurStevensonStorn23}. For our method this leads with solution $u\in U$ to \eqref{eq:Heat}, any approximation $u_h \in U_h$, and the Fortin operator $F \colon V \to V_h$ in Theorem~\ref{thm:FortinOperator} to the estimate
\begin{equation}\label{eq:Aposteriori}
\lVert u - u_h \rVert_U^2 \eqsim (Bu_h-f) (G_h(Bu_h-f)) + \lVert u_h(0) - u_0\rVert^2_{L^2(\Omega)} + \lVert f \circ (\identity - F)\rVert^2_{V'}.
\end{equation} 
The first two terms on the right-hand side can be computed.
The following lemma, involving the $L^2(K)$ orthogonal projection $P_K\colon L^2(K) \to \mathbb{P}_{p_t-1,p_x}(K) \oplus \mathbb{P}_{p_t,p_x-2}(K)$ for all $K\in \tria$, allows us to bound the third term by some higher-order data oscillation term.
\begin{lemma}[Data oscillation term]
For any right-hand side $f\in L^2(Q)$ one has
\begin{equation*}
\lVert f \circ (\identity - F)\rVert^2_{V'} \lesssim \sum_{K\in \tria} h_{K,x}^2 \lVert f - P_K f\rVert_{L^2(K)}^2.
\end{equation*}
\end{lemma}
\begin{proof}
Let $v\in V$. Then the design of the Fortin operator $F$ in Theorem~\ref{thm:FortinOperator} with the annihilation property in \eqref{eq:DefC2} yields
\begin{align*}
\int_Q f (v-F v) \dz &= \sum_{K\in \tria} \int_K (f-P_K f) (v-F v) \dz \\
&\leq  \Big(\sum_{K\in \tria} h_{K,x}^2 \lVert f - P_K f\rVert_{L^2(K)}^2 \Big)^{1/2} \Big(\sum_{K\in \tria} h_{K,x}^{-2} \lVert v - F v \rVert_{L^2(K)}^2 \Big)^{1/2}.
\end{align*}
The local approximation properties of the operator $F$ displayed in the proof of Theorem~\ref{thm:FortinOperator} show that the last contribution on the right-hand side is bounded by a multiple of $\lVert v \rVert_V$, concluding the proof.
\end{proof}
Suppose that $f\in L^2(Q)$ and define $w_h \coloneqq G_h(Bu_h-f)$. Then the localized contributions of the error estimator in \eqref{eq:Aposteriori}, without the data-oscillation term, read for all $K\in \tria$,
\begin{equation}\label{eq:LocErrorEst}
\begin{aligned}
\eta^2(u_h,K) \coloneqq \eta^2(K) &\coloneqq \int_K \partial_t u_h\, w_h \dz + \int_K \nabla_x u_h \cdot \nabla_x w_h\dz - \int_K fw_h \dz\\
&\quad  + \int_{K\cap \lbrace t = 0\rbrace } |u_h(0) - u_0|^2 \dx,
\end{aligned}
\end{equation}
that is, we have with $\eta^2(u_h,\tria) \coloneqq \eta^2(\tria)\coloneqq \sum_{K\in \tria}\eta^2(K)$ the equivalence
\begin{equation}\label{eq:defEta}
\lVert u - u_h \rVert_U^2 \eqsim \eta^2(\tria) + \lVert f \circ (\identity - F)\rVert^2_{V'}.
\end{equation}
Ignoring the data-oscillation term, we drive our adaptive scheme in Section~\ref{sec:NumExp} below by these local error indicators $\eta^2(u_h,K)$.
\begin{remark}[Negative error indicators]\label{rem:NegErrorInd}
Note that the error indicator $\eta^2(K)$ may take negative values for some $K\in \tria$. A resulting notable difference from typical adaptive schemes with non-negative error indicators is that bulk parameters $\theta$ close to or equal to one in our D\"orfler marking strategy might lead to the refinement of only a few elements. For example, if $( \eta^2(K) )_{K \in \tria} = (1, 0.1, 0.1, -0.1, -0.1)$ and $\theta = 1$, the smallest set $\mathcal{M}$ of simplices satisfying $\sum_{K \in \mathcal{M}} \eta^2(K) \leq \theta \sum_{K \in \tria} \eta^2(K)$ will contain only the element $K \in \tria$ with $\eta^2(K) = 1$. This appears indeed to be a problem in the numerical experiment in Section~\ref{subsec:ExpSmootSol} below.
We explored various modifications to enhance performance such as taking the absolute value $\tilde{\eta}^2(K) \coloneqq |\eta^2(K)|$. We found that the best results were achieved by discarding negative values, leading to the definition $\tilde{\eta}^2(K) \coloneqq \max\lbrace 0, \eta^2(K)\rbrace$. Addressing this heuristic adjustment through the introduction of localizable error estimators remains an avenue for future research.
\end{remark}
\subsection{Stopping criterion}\label{subsec:StoppingCriterion}
Our PCG-scheme computes iterates $(u_h^n)_{n\in \mathbb{N}_0} \subset U_h$ that converge to the solution $u_h \in U_h$ to the discretized problem in~\eqref{eq:MatrixProb}. To improve the efficiency of our computation, we want to stop the iterative scheme if our current iterate $u_h^n \in U_h$ satisfies
\begin{equation*}
\frac{\lVert u_h - u_h^n \rVert_U}{\lVert u - u_h^n \rVert_U} \leq 1/2,
\end{equation*}
since then the triangle inequality verifies the quasi-optimality result
\begin{equation*}
\lVert u - u_h^n \rVert_U \leq 2\, \lVert u - u_h \rVert_U \lesssim \min_{w_h\in U_h}\lVert u - w_h \rVert_U.
\end{equation*}
Due to \eqref{eq:Aposteriori} we can control the term $\lVert u - u_h^n \rVert^2_U$ by 
\begin{equation*}
\begin{aligned}
\eta^2(u_h^n,\tria) + \lVert f \circ (\identity - F)\rVert^2_{V'} \eqsim \lVert u - u_h^{n} \rVert^2_U.
\end{aligned}
\end{equation*}

Pretending that our preconditioner $K_h$ satisfies the norm equivalence~\eqref{eq:EquiKh} (which is only nearly true\footnote{Taking into account \eqref{eq:inequalities}, \eqref{eq:Defdres} will read as $1 \lesssim \frac{\mathtt{alg\_est}(u_h^n)^2}{\lVert u_h - u_h^n \rVert_U^2} \lesssim h_{t,\min}^{-2\varsigma}$}, we can use it to estimate the norm of the algebraic error $\lVert u_h - u_h^n \rVert_U$ as follows:
The definition of the inner product $(K_h^{-1}\bigcdot)(\bigcdot)$ on $U_h \times U_h$ shows that $K_h$ is the inverse Riesz map associated to the Hilbert space $\big(U_h,(K_h^{-1}\bigcdot)(\bigcdot)\big)$. Indeed, we have $g(\bigcdot)=(K_h^{-1}K_h g)(\bigcdot)$ for all $g \in U_h'$.
From the Riesz map being an isometric isomorphism, we infer that
\begin{equation*}
\sup_{0 \neq v \in U_h}\frac{g(v)^2}{(K_h^{-1}v)(v)}=(K_h^{-1} K_h g)(K_h g).
\end{equation*}
Hence, by substituting $g=A_h w$, we obtain
\begin{equation*}
(A_h w)(K_h A_h w) =\sup_{0 \neq v \in U_h}\frac{(A_h w)(v)^2}{(K_h^{-1}v)(v)} \eqsim \sup_{0 \neq v \in U_h}\frac{(A_h w)(v)^2}{(A_h v)(v)}= (A_h w)(w) \eqsim \|w\|_U^2.
\end{equation*}
For $w=u_h - u_h^n$, this yields the computable error estimator
\begin{equation}\label{eq:Defdres}
\lVert u_h - u_h^n \rVert_U^2 \eqsim (F_h-A_h u_h^n)(K_h(F_h-A_h u_h^n))\eqqcolon \mathtt{alg\_est}(u_h^n)^2.
\end{equation}%
We thus terminate the PCG-scheme if the current iterate $u_h^n\in U_h$ satisfies for some fixed positive value $\varepsilon \ll 1$ the stopping criterion 
\begin{equation}\label{eq:StoppCrit}
\frac{\lVert u_h - u_h^n \rVert^2_U}{\lVert u - u_h^n \rVert^2_U} \eqsim \frac{\mathtt{alg\_est}(u_h^n)^2}{\eta^2(u_h^n,\tria) + \lVert f \circ (\identity - F)\rVert^2_{V'} } \leq \varepsilon.
\end{equation}

\section{Numerical experiments}\label{sec:NumExp}
We conclude our investigation by conducting a numerical study of the proposed method. We use the ansatz space $U_h = W_h^{p_t, p_x}$, as defined in \eqref{eq:DefW_h0ptpx}, with polynomial degrees $p_t = 1$ and $p_x = 1$ or $p_x = 3$, where the latter choice follows the suggestion in Remark~\ref{rem:1}.
Our test spaces read $V_h = W_h^{p_t+2, p_x+3} \cap V$, which includes the space defined in \eqref{eq:testSpace} verifying discrete inf-sup stability.
Unlike for the space in \eqref{eq:testSpace}, this definition of $V_h$ leads to nested spaces under mesh refinement, allowing for the application of the additive  subspace correction in Section~\ref{subsec:AddSubspaceCorr} as preconditioner without exploiting the split discussed in Section~\ref{subsec:Splitting}. Our implementation, however, uses the split $V_h = V_{h,0} \oplus \mathcal{B}_h$ with $\mathcal{B}_h = \lbrace v_h \in V_h\colon v_h(j) = 0\text{ for all vertices } j \text{ in }\tria\rbrace$ and the corresponding preconditioner in \eqref{eq:SplittedPrecond}.

We employ the PCG-scheme detailed in Section~\ref{subsec:PCGscheme}, using the stopping criterion from \eqref{eq:StoppCrit} with a tolerance of $\varepsilon = 0.01$.
The initial iterate $u_h^0$ equals zero on the initial mesh and equals the final iterate from the previous coarser mesh for refined meshes.
Our adaptive mesh refinement scheme is driven by the error estimator $\eta^2(\tria) = \sum_{K\in \tria}\eta^2(K)$ in \eqref{eq:LocErrorEst} and the D\"orfler marking strategy with bulk parameter $\theta = 0.5$.

\subsection{Smooth solution}\label{subsec:ExpSmootSol}
Our first experiment investigates the equivalence of the computable residual $\eta(\tria) \coloneqq \eta^2(\tria)^{1/2}$ and the error. Since we cannot directly evaluate the error in the norm $\lVert \bigcdot \rVert_U$, we instead focus on the contribution $\lVert \nabla_x (u-u_h) \rVert_{L^2(Q)} \leq \sqrt{2}\, \lVert u-u_h \rVert_U$.
The domain $\Omega = (0,1)^2$ is the unit square and the time interval equals $\cJ = (0,1)$. The exact solution is given by
\begin{equation*}
u(t,x,y) = t(x-x^2)(y-y^2) \qquad \text{for all } (t,x,y) \in Q = \cJ \times \Omega.
\end{equation*}
We solve the problem using both uniform and adaptive mesh refinements. The corresponding convergence history is depicted in Figure~\ref{fig:smoothSol}, demonstrating a favorable efficiency index for the residual $\eta(\tria)$.

As noted in Remark~\ref{rem:1}, we anticipate convergence rates of $\mathcal{O}(\textup{ndof}^{-1/4})$ and $\mathcal{O}(\textup{ndof}^{-3/4})$ for $p_x = 1$ and $p_x = 3$, respectively. The convergence plots confirm this expected rate for $p_x=3$. For $p_x=1$ the rate is slightly worse, which might be caused by the fact that the equivalence constants in $(G_h^{-1} \bigcdot)(\bigcdot) \eqsim \|\bigcdot\|_V^2$ on $V_h$, although being uniformly bounded, are still pre-asymptotically dependent on the underlying meshes. In particular multi-level preconditioners are known to be better on coarser meshes than on finer ones, although the difference `saturates' when the meshes get finer. 

The adaptive refinement strategy exhibits challenges for $p_x = 3$ 
in the sense that in each adaptive loop only a small number of elements is marked for refinement. 
As discussed in Remark~\ref{rem:NegErrorInd}, this is due to negative contributions from the error estimator, resulting in the small number of elements being refined. This observation suggests the potential need for either an alternative error estimator or a modification of the marking strategy.
\begin{figure}
\begin{tikzpicture}
\begin{axis}[
clip=false,
width=.5\textwidth,
height=.45\textwidth,
ymode = log,
xmode = log,
cycle multi list={\nextlist MyColors},
scale = {1},
clip = true,
legend cell align=left,
legend style={legend columns=1,legend pos= south west,font=\fontsize{7}{5}\selectfont}
]
	\addplot table [x=ndof,y=res] {experiments/new/ST_KnownSol_unif_px_1.txt};
	\addplot table [x=ndof,y=error] {experiments/new/ST_KnownSol_unif_px_1.txt};
	\addplot table [x=ndof,y=res] {experiments/new/ST_KnownSol_adapt_px_1.txt};
	\addplot table [x=ndof,y=error] {experiments/new/ST_KnownSol_adapt_px_1.txt};
	\legend{{$\eta(\tria)$},{$\lVert \nabla_x (u{-}u_h)\rVert_{L^2(Q)}$}};
	\addplot[dashed, sharp plot,update limits=false] coordinates {(1e1,6e-2) (1e5,6e-3)};
\end{axis}
\end{tikzpicture}
\begin{tikzpicture}
\begin{axis}[
clip=false,
width=.5\textwidth,
height=.45\textwidth,
ymode = log,
xmode = log,
cycle multi list={\nextlist MyColors},
scale = {1},
clip = true,
legend cell align=left,
legend style={legend columns=1,legend pos= south west,font=\fontsize{7}{5}\selectfont}
]
	\addplot table [x=ndof,y=res] {experiments/new/ST_KnownSol_unif_px_3.txt};
	\addplot table [x=ndof,y=error] {experiments/new/ST_KnownSol_unif_px_3.txt};
	\addplot table [x=ndof,y=res] {experiments/new/ST_KnownSol_adapt_px_3.txt};
	\addplot table [x=ndof,y=error] {experiments/new/ST_KnownSol_adapt_px_3.txt};	
	\legend{{$\eta(\tria)$},{$\lVert \nabla_x (u-{u}_h)\rVert_{L^2(Q)}$}};
	\addplot[dashdotted, sharp plot,update limits=false] coordinates {(1e2,3e-3) (1e6,3e-6)};
\end{axis}
\end{tikzpicture}
\caption{Convergence of the residual $\eta(\tria)$ and the error $\lVert \nabla_x (u-u_h)\rVert_{L^2(Q)}$ plotted against $\textup{ndof} \coloneqq \dim U_h$ with uniform (solid line) and adaptive (dotted line) mesh refinements with polynomial degree $p_x = 1$ (left) and $p_x = 3$ (right) in the experiment in Section~\ref{subsec:ExpSmootSol} (Smooth solution). The dashed line (left) indicates the slope $\textup{ndof}^{-3/4}$ and the dash-dotted line (right) the slope $\textup{ndof}^{-1/4}$.}\label{fig:smoothSol}
\end{figure}
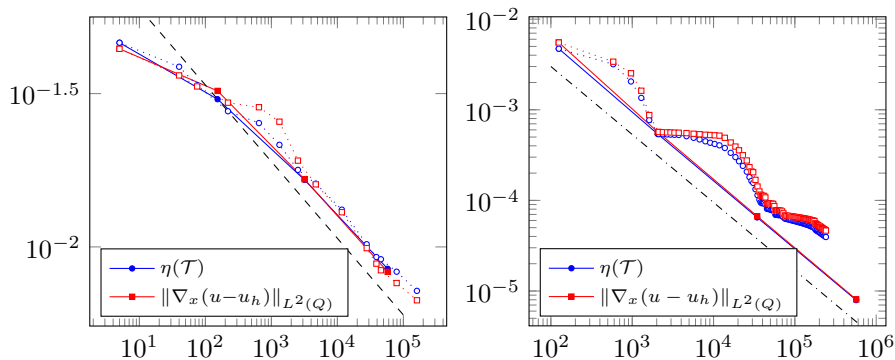
\subsection{L-shaped domain}\label{subsec:ExpLShape}
Our second experiment addresses the heat equation on the L-shaped domain $\Omega = (-1,1)^2 \setminus [0,1)^2$ over the time interval $\cJ = (0,1)$, with constant right-hand side $f = 1$ and initial condition $u_0 = 0$. The resulting convergence history of the residual $\eta(\tria)$ is shown in Figure~\ref{fig:ExpLShape}. Interestingly, the adaptive refinement strategy for $p_x = 1$ yields the same convergence rate as the uniform mesh refinement, specifically $\textup{ndof}^{-1/5}$. In fact, the meshes generated by adaptive refinement (not shown in this paper) appear to be nearly uniform. 
For $p_x = 3$, we achieve a much better approximation. The adaptively refined meshes, displayed in Figure~\ref{fig:ExpLshapeMesh}, capture the corner singularity at the origin of $\Omega$, leading to an improved convergence rate that appears to behave better than $\textup{ndof}^{-1/3}$.

\begin{figure}
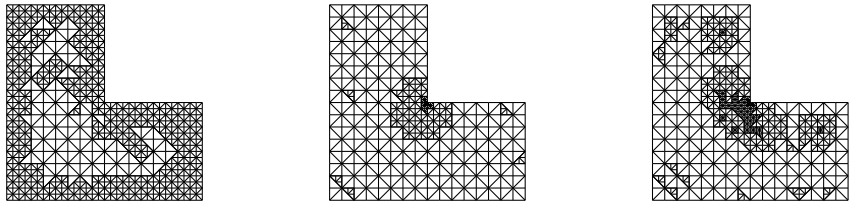

{\centering
\hbox{
\begin{minipage}{.03\textwidth}
\ 
\end{minipage}
\begin{minipage}{0.33\textwidth}
\include{experiments/RefLShapeT0.tex}
\end{minipage}
\begin{minipage}{0.33\textwidth}
\include{experiments/RefLShapeThalf.tex}
\end{minipage}
\begin{minipage}{0.33\textwidth}
\include{experiments/RefLShapeTend.tex}
\end{minipage}
}}
\vspace*{-.8cm}
\caption{Meshes at $t = 0$ (left), $t=0.5$ (center), and $t=1$ (right) at level $\ell = 17$ ($\textup{ndof} = 752\, 258$) of the adaptive mesh refinement routine with $p_x = 3$.}\label{fig:ExpLshapeMesh}
\end{figure}

\begin{figure}
\begin{tikzpicture}
\begin{axis}[
clip=false,
width=.7\textwidth,
height=.45\textwidth,
ymode = log,
xmode = log,
cycle multi list={\nextlist MyColors},
scale = {1},
clip = true,
legend cell align=left,
legend style={legend columns=1,legend pos= north east,font=\fontsize{7}{5}\selectfont}
]
	\addplot table [x=ndof,y=res] {experiments/new/ST_Lshape_unif_px1.txt};
	\addplot table [x=ndof,y=res] {experiments/new/ST_Lshape_unif_px3.txt};
	\addplot table [x=ndof,y=res] {experiments/new/ST_Lshape_adapt_px1.txt};
	\addplot table [x=ndof,y=res] {experiments/new/ST_Lshape_adapt_px3.txt};
	\addplot[dashdotted, sharp plot,update limits=false] coordinates {(2e1,2e-1) (2e6,2e-2)};
	\addplot[dashed, sharp plot,update limits=false] coordinates {(1e3,2.5e-2) (1e6,2.5e-3)};
	\legend{{$p_x=1$},{$p_x=3$}};
\end{axis}
\end{tikzpicture}
\caption{Convergence history plot of the residual $\eta(\tria)$ in the experiment of Section~\ref{subsec:ExpLShape} (L-shaped domain) plotted against $\textup{ndof} \coloneqq \dim U_h$ with uniform (solid line) and adaptive (dotted line) mesh refinements. The dash-dotted line indicates the slop $\textup{ndof}^{-1/5}$ and the dashed line indicates the slope $\textup{ndof}^{-1/3}$.}\label{fig:ExpLShape}
\end{figure}
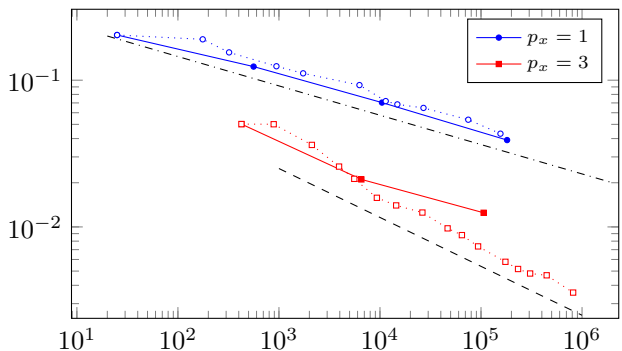
\subsection{Rough initial data}\label{subsec:ExpRoughInit}
This experiment examines the approximation of the heat equation with homogeneous right-hand side $f = 0$ and non-matching initial data $u_0 = 1$, defined over the domain $\Omega = (0,1)^2$ and time interval $\cJ = (0,1)$. The solution $u$ exhibits a pronounced singularity at the boundary $\partial \Omega$ at initial time $t=0$. 
Existing adaptive time-space methods face considerable challenges with this benchmark problem. Specifically, the adaptive schemes proposed in \cite{FuehrerKarkulik21} yield convergence rates of $\textup{ndof}^{-0.07}$ and $\textup{ndof}^{-0.08}$ for uniform and adaptive mesh refinements, respectively. Similarly, the scheme in \cite{GantnerStevenson23} achieve rates of $\textup{ndof}^{-0.09}$ and $\textup{ndof}^{-0.14}$ for uniform and adaptive refinements, respectively. 
In contrast, the convergence history shown in Figure~\ref{fig:ExpInitOne} demonstrates a remarkable improvement using our novel approach.
This illustrates the advantages when using parabolic scaling and furthermore indicates that the norms used in the papers mentioned before are too strong for these kind of highly singular solutions. We obtain convergence rates of $\textup{ndof}^{-1/8}$ and $\textup{ndof}^{-1/4}$ for $p_x = 1$ with uniform and adaptive refinements, respectively, and rates of $\textup{ndof}^{-1/8}$ and $\textup{ndof}^{-2/5}$ for $p_x = 3$. Notably, the rate achieved by the adaptive scheme for $p_x=1$ appears to be optimal.

\tdplotsetmaincoords{120}{40} 
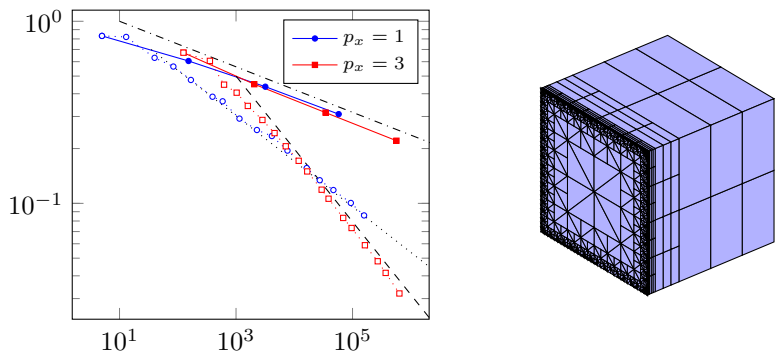
\begin{figure}
\begin{subfigure}[b]{0.5\textwidth}
\begin{tikzpicture}
\begin{axis}[
clip=false,
width= 1\textwidth,
height=.9\textwidth,
ymode = log,
xmode = log,
cycle multi list={\nextlist MyColors},
scale = {1},
clip = true,
legend cell align=left,
legend style={legend columns=1,legend pos= north east,font=\fontsize{7}{5}\selectfont}
]
	\addplot table [x=ndof,y=res] {experiments/new/ST_InitOne_unif_px1.txt};
	\addplot table [x=ndof,y=res] {experiments/new/ST_InitOne_unif_px3.txt};
	\addplot table [x=ndof,y=res] {experiments/new/ST_InitOne_adapt_px1.txt};
	\addplot table [x=ndof,y=res] {experiments/new/ST_InitOne_adapt_px3.txt};
	\addplot[dashdotted, sharp plot,update limits=false] coordinates {(1e1,1e0) (1e9,1e-1)};
	\addplot[dashed, sharp plot,update limits=false] coordinates {(1e3,5e-1) (1e8,5e-3)};
	\addplot[dotted, sharp plot,update limits=false] coordinates {(3.5e1,7e-1) (3.5e9,7e-3)};
		\legend{{$p_x=1$},{$p_x=3$}};
\end{axis}
\end{tikzpicture}
\end{subfigure}
\begin{subfigure}[b]{0.05\textwidth}
\ 
\end{subfigure}
\begin{subfigure}[b]{0.4\textwidth}
\include{experiments/corrected/MeshInitOne.tex}
\end{subfigure}
\caption{Left. Convergence history of the residual $\eta(\tria)$ plotted against $\textup{ndof} = \dim U_h$ for the experiment in Section~\ref{subsec:ExpRoughInit} (Rough initial data) with uniform (solid) and adaptive (dotted) refinements. The dashed line indicates the rate $\textup{ndof}^{-2/5}$, the dash-dotted line $\textup{ndof}^{-1/8}$, and the dotted line $\textup{ndof}^{-1/4}$.\\
Right. The surface of the mesh resulting from the experiment with $p_x = 3$ and $\textup{ndof} = 13\, 688$.}\label{fig:ExpInitOne}
\end{figure}

\subsection{Fundamental solution}\label{subsec:expFundSol}
This experiment aims at the approximation of the fundamental solution to the heat equation 
\begin{equation}\label{eq:FundSol}
u(t,x) = \frac{1}{4\pi (t+t_s)} \exp\left(-\frac{|x|}{4(t+t_s)}\right)   \qquad\text{with time shift }t_s \coloneqq 10^{-3}.
\end{equation}
We included the time shift $t_s$ to obtain initial data $u_0 \coloneqq u(0,\bigcdot)$ in $L^2(\Omega)$. 
The right-hand side reads $f =0$ and we use the inhomogeneous Dirichlet boundary condition dictated by \eqref{eq:FundSol} on the lateral boundary of the time-space cylinder $Q = \cJ \times \Omega$ with $\cJ = (0,1)$ and $\Omega = (-1,1)^2$. Figure~\ref{fig:fundSol} displays the resulting convergence history plot of the residuals. As for the experiment in Section~\ref{subsec:ExpSmootSol}, our marking strategy causes the refinement of solely a few elements. We remedied this by using the contribution $\tilde{\eta}^2(K) \coloneqq \max\lbrace \eta^2(K),0\rbrace$ for all $K\in \tria$ as refinement indicator in the D\"orfler marking strategy, leading to residuals illustrated by dashed lines.

The adaptive refinement leads to significantly improved results. However, the rate for $p_x=1$ seems to be slightly worse that $\textup{ndof}^{-1/4}$ which should be the optimal rate. This can be a pre-asymptotic effect caused by the fact that the quadrature of the initial data $u_0$ improves as the mesh is refined. For $p_x =3$ the rate of convergence seems to be approximately twice as good, which is similar to the rate in the experiment of Section~\ref{subsec:ExpRoughInit}.
 
\begin{figure}
\begin{tikzpicture}
\begin{axis}[
clip=false,
width=.7\textwidth,
height=.45\textwidth,
ymode = log,
xmode = log,
cycle multi list={\nextlist MyColors},
scale = {1},
clip = true,
legend cell align=left,
legend style={legend columns=1,legend pos= outer north east,font=\fontsize{7}{5}\selectfont}
]

	\addplot table [x=ndof,y=res] {experiments/new/ST_Dirac_unif_px1.txt};
	\addplot table [x=ndof,y=res] {experiments/new/ST_Dirac_unif_px3.txt};
	\addplot table [x=ndof,y=res] {experiments/new/ST_Dirac_adapt_px1.txt};
	\addplot table [x=ndof,y=res] {experiments/new/ST_Dirac_adapt_px3.txt};
	\addplot table [x=ndof,y=res] {experiments/new/ST_Dirac_adapt_noZeros_px1.txt};		
	\addplot table [x=ndof,y=res] {experiments/new/ST_Dirac_adapt_noZeros_px3.txt};	
	\addplot[dashed, sharp plot,update limits=false] coordinates {(1e3,2e0) (1e7,2e-1)};	
	\addplot[dotted, sharp plot,update limits=false] coordinates {(1e3,1e0) (1e7,1e-2)};	
	
	\legend{{$p_x=1$},{$p_x=3$}};
	
\end{axis}
\end{tikzpicture}
\caption{Convergence history plot of the residual $\eta(\tria)$ in the experiment of Section~\ref{subsec:expFundSol} (Fundamental solution) with uniform (solid) and adaptive refinements with error indicator $\eta^2(K)$ (dotted) and $\tilde{\eta}^2(K) \coloneqq \max\lbrace 0 ,\eta^2(K)\rbrace$ (dashed). The dashed line indicates the rate $\textup{ndof}^{-1/4}$ and the dotted $\textup{ndof}^{-1/2}$.}\label{fig:fundSol}
\end{figure}
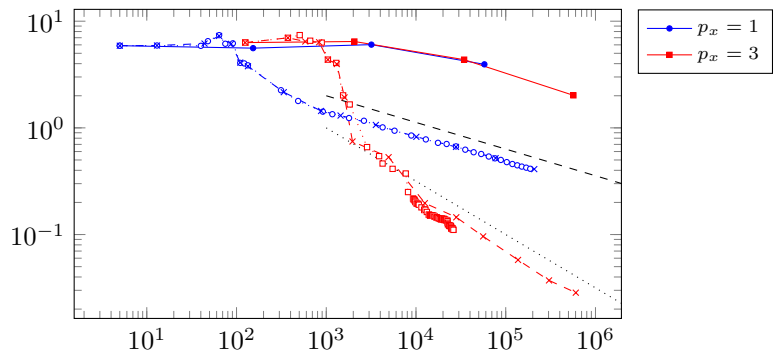

\subsection{PCG-scheme}\label{subsec:ExpPreCond}
We conclude our numerical study by analyzing the number of iterations required by the preconditioned conjugate gradient method with iterates $(u_h^n)_{n\in \mathbb{N}_0} \subset U_h$. 
Figure~\ref{fig:nrCGiter} illustrates the number of PCG-iterations needed for the computations presented in Section~\ref{subsec:ExpLShape} and Section~\ref{subsec:ExpRoughInit} with initial value $u_h^0$ being the solution on the previous mesh. For uniform mesh refinements, the results indicate that the number of PCG-iterations remains uniformly bounded or, at most, grows very slowly. In contrast, the number of PCG-iterations under the adaptive refinement strategy appears to grow. This behavior may be a pre-asymptotic effect, since multi-level preconditioners are known to be better on coarser meshes than on finer ones. Recalling that our preconditioner $G_h\colon V_h \rightarrow V_h'$  is defined as $E_{h,0} G_{h,0} E_{h,0}'+E_{h,1} G_{h,1} E_{h,1}'$, a possible remedy of this pre-asympotic effect can be a proper scaling the multi-level preconditioner $G_{h,0}$. Indeed, the dashed line in Figure~\ref{fig:nrCGiter} displays the number of PCG-iterations when $G_{h,0}$ is replaced by its weighted version $0.1\cdot G_{h,0}$, leading in the experiment of Section~\ref{subsec:ExpLShape} with $p_x=3$ to a seemingly uniformly bounded number of iterations.
\begin{figure}
\begin{tikzpicture}
\begin{axis}[
clip=false,
width= .5\textwidth,
height=.45\textwidth,
ymode = log,
xmode = log,
cycle multi list={\nextlist MyColors},
scale = {1},
clip = true,
legend cell align=left,
legend style={legend columns=1,legend pos= north west,font=\fontsize{7}{5}\selectfont}
]
	\addplot table [x=ndof,y=nrCGiter] {experiments/corrected/ST_Lshape_unif_px1.txt};
	\addplot table [x=ndof,y=nrCGiter] {experiments/corrected/ST_Lshape_unif_px3.txt};
		\addplot table [x=ndof,y=nrCGiter] {experiments/corrected/ST_Lshape_adapt_px1.txt};
	\addplot table [x=ndof,y=nrCGiter] {experiments/corrected/ST_Lshape_adapt_px3.txt};
		\addplot table [x=ndof,y=nrCGiter] {experiments/corrected/ST_Lshape_scaled_adapt_px1.txt};
	\addplot table [x=ndof,y=nrCGiter] {experiments/corrected/ST_Lshape_scaled_adapt_px3.txt};	

		\legend{{$p_x=1$},{$p_x=3$}};
\end{axis}
\end{tikzpicture}
\begin{tikzpicture}
\begin{axis}[
clip=false,
width= .5\textwidth,
height=.45\textwidth,
ymode = log,
xmode = log,
cycle multi list={\nextlist MyColors},
scale = {1},
clip = true,
legend cell align=left,
legend style={legend columns=1,legend pos= north west,font=\fontsize{7}{5}\selectfont}
]
	\addplot table [x=ndof,y=nrCGiter] {experiments/new/ST_InitOne_unif_px1.txt};
	\addplot table [x=ndof,y=nrCGiter] {experiments/new/ST_InitOne_unif_px3.txt};
		\addplot table [x=ndof,y=nrCGiter] {experiments/new/ST_InitOne_adapt_px1.txt};
	\addplot table [x=ndof,y=nrCGiter] {experiments/new/ST_InitOne_adapt_px3.txt};		
	\addplot table [x=ndof,y=nrCGiter] {experiments/new/ST_InitOne_scaled_adapt_px1.txt};
	\addplot table [x=ndof,y=nrCGiter] {experiments/new/ST_InitOne_scaled_adapt_px3.txt};	

		\legend{{$p_x=1$},{$p_x=3$}};
\end{axis}
\end{tikzpicture}
\caption{Number of PCG-iterations needed to satisfy the stopping criterion in the experiment in Section~\ref{subsec:ExpLShape} (left) and Section~\ref{subsec:ExpRoughInit} (right) with uniform (solid line) and adaptive (dotted line) mesh refinements. The dashed line results from adaptive mesh refinements and the replacement of the multi-level preconditioner $G_{h,0}$ by a scaled version $0.1\cdot G_{h,0}$.}\label{fig:nrCGiter}
\end{figure}
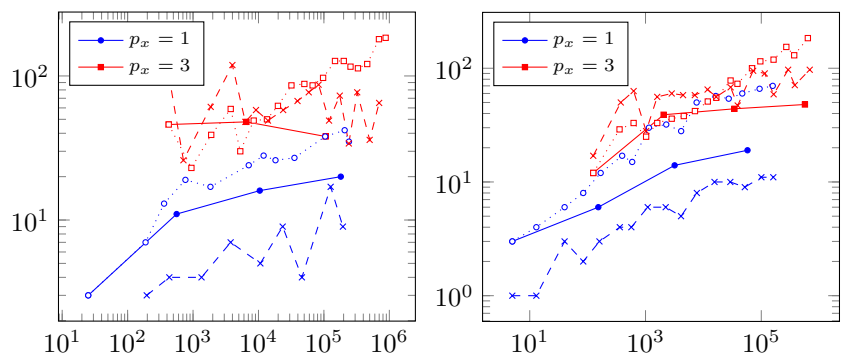


\printbibliography

@book {MR2424078,
    AUTHOR = {Adams, Robert A. and Fournier, John J. F.},
     TITLE = {Sobolev spaces},
    SERIES = {Pure and Applied Mathematics (Amsterdam)},
    VOLUME = {140},
   EDITION = {Second},
 PUBLISHER = {Elsevier/Academic Press, Amsterdam},
      YEAR = {2003},
     PAGES = {xiv+305},
      ISBN = {0-12-044143-8},
   MRCLASS = {46E35 (46-01 46-02 46B70 46Exx)},
  MRNUMBER = {2424078},
}

@article {Andreev16,
    AUTHOR = {Andreev, Roman},
     TITLE = {Wavelet-in-time multigrid-in-space preconditioning of
              parabolic evolution equations},
   JOURNAL = {SIAM J. Sci. Comput.},
  FJOURNAL = {SIAM Journal on Scientific Computing},
    VOLUME = {38},
      YEAR = {2016},
    NUMBER = {1},
     PAGES = {A216--A242},
      ISSN = {1064-8275,1095-7197},
   MRCLASS = {65M60 (35K90 65M12 65M20 65M22 65M55 65Y05)},
  MRNUMBER = {3449910},
MRREVIEWER = {Rafael\ F.\ Santos},
       DOI = {10.1137/140998639},
       URL = {https://doi.org/10.1137/140998639},
}

@article {BaiocchiBrezzi83,
    AUTHOR = {Baiocchi, C. and Brezzi, F.},
     TITLE = {Optimal error estimates for linear parabolic problems under
              minimal regularity assumptions},
   JOURNAL = {Calcolo},
  FJOURNAL = {Calcolo. A Quarterly on Numerical Analysis and Theory of
              Computation},
    VOLUME = {20},
      YEAR = {1983},
    NUMBER = {2},
     PAGES = {143--176},
      ISSN = {0008-0624},
   MRCLASS = {65M15},
  MRNUMBER = {746351},
MRREVIEWER = {I.\ P.\ Gavrilyuk},
       DOI = {10.1007/BF02575590},
       URL = {https://doi.org/10.1007/BF02575590},
}

@article {BinevDeVore04,
    AUTHOR = {Binev, Peter and DeVore, Ronald},
     TITLE = {Fast computation in adaptive tree approximation},
   JOURNAL = {Numer. Math.},
  FJOURNAL = {Numerische Mathematik},
    VOLUME = {97},
      YEAR = {2004},
    NUMBER = {2},
     PAGES = {193--217},
      ISSN = {0029-599X,0945-3245},
   MRCLASS = {65Y99 (68W25 68W40)},
  MRNUMBER = {2050076},
       DOI = {10.1007/s00211-003-0493-6},
       URL = {https://doi.org/10.1007/s00211-003-0493-6},
}

@book {BrennerScott08,
    AUTHOR = {Brenner, S. C. and Scott, L. R.},
     TITLE = {The mathematical theory of finite element methods},
    SERIES = {Texts in Applied Mathematics},
    VOLUME = {15},
   EDITION = {Third},
 PUBLISHER = {Springer, New York},
      YEAR = {2008},
     PAGES = {xviii+397},
      ISBN = {978-0-387-75933-3},
   MRCLASS = {65-01 (65-02)},
  MRNUMBER = {2373954},
       DOI = {10.1007/978-0-387-75934-0},
       URL = {https://doi.org/10.1007/978-0-387-75934-0},
}

@article{CarstensenFaermann01,
        TITLE = {Mathematical foundation of a~posteriori error estimates and adaptive mesh-refining algorithms for boundary integral equations of the first kind},
       AUTHOR = {Carstensen, C. and Faermann, B.},
      JOURNAL = {EABE},
     FJOURNAL = {Journal of Engineering Analysis with Boundary Elements},
         YEAR = {2001},
       VOLUME = {25},
       NUMBER = {7},
        PAGES = {497--509},
Xbtbrowserfulltext = {../download/2001-CC_FB-A_Posteriori_Error_Adaptive_Mesh-Refining_Boundary_Integral_Equations.pdf},
}

@book {DautrayLions92,
    AUTHOR = {Dautray, Robert and Lions, Jacques-Louis},
     TITLE = {Mathematical analysis and numerical methods for science and
              technology. {V}ol. 5},
      NOTE = {Evolution problems. I,
              With the collaboration of Michel Artola, Michel Cessenat and
              H\'{e}l\`ene Lanchon,
              Translated from the French by Alan Craig},
 PUBLISHER = {Springer-Verlag, Berlin},
      YEAR = {1992},
     PAGES = {xiv+709},
      ISBN = {3-540-66101-8},
   MRCLASS = {00A05 (35-01 47-01)},
  MRNUMBER = {1156075},
       DOI = {10.1007/978-3-642-58090-1},
       URL = {https://doi.org/10.1007/978-3-642-58090-1},
}

@article {DahmenStevenson99,
    AUTHOR = {Dahmen, Wolfgang and Stevenson, Rob},
     TITLE = {Element-by-element construction of wavelets satisfying
              stability and moment conditions},
   JOURNAL = {SIAM J. Numer. Anal.},
  FJOURNAL = {SIAM Journal on Numerical Analysis},
    VOLUME = {37},
      YEAR = {1999},
    NUMBER = {1},
     PAGES = {319--352},
      ISSN = {0036-1429,1095-7170},
   MRCLASS = {65N30 (42C40 65F35 65T60)},
  MRNUMBER = {1742747},
MRREVIEWER = {Tian-Xiao\ He},
       DOI = {10.1137/S0036142997330949},
       URL = {https://doi.org/10.1137/S0036142997330949},
}

@article {DanwitzVoulisHostersBehr23,
    AUTHOR = {von Danwitz, Max and Voulis, Igor and Hosters, Norbert and
              Behr, Marek},
     TITLE = {Time-continuous and time-discontinuous space-time finite
              elements for advection-diffusion problems},
   JOURNAL = {Internat. J. Numer. Methods Engrg.},
  FJOURNAL = {International Journal for Numerical Methods in Engineering},
    VOLUME = {124},
      YEAR = {2023},
    NUMBER = {14},
     PAGES = {3117--3144},
      ISSN = {0029-5981,1097-0207},
   MRCLASS = {65M60 (76R50)},
  MRNUMBER = {4601590},
       DOI = {10.1002/nme.7241},
       URL = {https://doi.org/10.1002/nme.7241},
}

@article{DieningGehringStorn23,
      title={Adaptive {M}esh {R}efinement for arbitrary initial {T}riangulations}, 
      author={Lars Diening and Lukas Gehring and Johannes Storn},
      year={2023},
        publisher = {arXiv},
  year = {2023},
  journal = {arXiv},
  DOI = {10.48550/arXiv.2306.02674},
}

@article {DieningStorn22,
    AUTHOR = {Diening, Lars and Storn, Johannes},
     TITLE = {A space-time {DPG} method for the heat equation},
   JOURNAL = {Comput. Math. Appl.},
  FJOURNAL = {Computers \& Mathematics with Applications. An International
              Journal},
    VOLUME = {105},
      YEAR = {2022},
     PAGES = {41--53},
      ISSN = {0898-1221},
   MRCLASS = {65M60 (65M50)},
  MRNUMBER = {4345259},
MRREVIEWER = {Yufeng Nie},
       DOI = {10.1016/j.camwa.2021.11.013},
       URL = {https://doi.org/10.1016/j.camwa.2021.11.013},
}

@book {ErnGuermond04,
    AUTHOR = {Ern, Alexandre and Guermond, Jean-Luc},
     TITLE = {Theory and practice of finite elements},
    SERIES = {Applied Mathematical Sciences},
    VOLUME = {159},
 PUBLISHER = {Springer-Verlag, New York},
      YEAR = {2004},
     PAGES = {xiv+524},
      ISBN = {0-387-20574-8},
   MRCLASS = {65-02 (65M60 65N30 74S05 76M10 78M10)},
  MRNUMBER = {2050138},
MRREVIEWER = {R. S. Anderssen},
       DOI = {10.1007/978-1-4757-4355-5},
       URL = {https://doi.org/10.1007/978-1-4757-4355-5},
}

@book {ErnGuermond21c,
    AUTHOR = {Ern, Alexandre and Guermond, Jean-Luc},
     TITLE = {Finite elements {III}---{F}irst-order and time-dependent {PDE}s},
    SERIES = {Texts in Applied Mathematics},
    VOLUME = {74},
 PUBLISHER = {Springer, Cham},
      YEAR = {2021},
     PAGES = {viii+417},
      ISBN = {978-3-030-57347-8},
   MRCLASS = {65-01},
  MRNUMBER = {4248810},
       DOI = {10.1007/978-3-030-57348-5},
       URL = {https://doi.org/10.1007/978-3-030-57348-5},
}

@article {Fontes09,
    AUTHOR = {Fontes, Magnus},
     TITLE = {Initial-boundary value problems for parabolic equations},
   JOURNAL = {Ann. Acad. Sci. Fenn. Math.},
  FJOURNAL = {Annales Academi\ae Scientiarum Fennic\ae . Mathematica},
    VOLUME = {34},
      YEAR = {2009},
    NUMBER = {2},
     PAGES = {583--605},
      ISSN = {1239-629X,1798-2383},
   MRCLASS = {35K60 (35K20 35K92 46E35 46F05)},
  MRNUMBER = {2553815},
MRREVIEWER = {Sabir\ Umarov},
}

@article {Fortin77,
    AUTHOR = {Fortin, Michel},
     TITLE = {An analysis of the convergence of mixed finite element
              methods},
   JOURNAL = {RAIRO Anal. Num\'{e}r.},
  FJOURNAL = {RAIRO Analyse Num\'{e}rique},
    VOLUME = {11},
      YEAR = {1977},
    NUMBER = {4},
     PAGES = {341--354, iii},
      ISSN = {0399-0516},
   MRCLASS = {65D05},
  MRNUMBER = {464543},
       DOI = {10.1051/m2an/1977110403411},
       URL = {https://doi.org/10.1051/m2an/1977110403411},
}

@article {FuehrerKarkulik21,
    AUTHOR = {F\"uhrer, Thomas and Karkulik, Michael},
     TITLE = {Space-time least-squares finite elements for parabolic
              equations},
   JOURNAL = {Comput. Math. Appl.},
  FJOURNAL = {Computers \& Mathematics with Applications. An International
              Journal},
    VOLUME = {92},
      YEAR = {2021},
     PAGES = {27--36},
      ISSN = {0898-1221,1873-7668},
   MRCLASS = {65M60 (65M12 65M15)},
  MRNUMBER = {4242919},
       DOI = {10.1016/j.camwa.2021.03.004},
       URL = {https://doi.org/10.1016/j.camwa.2021.03.004},
}

@article {GantnerStevenson21,
    AUTHOR = {Gantner, Gregor and Stevenson, Rob},
     TITLE = {Further results on a space-time {FOSLS} formulation of
              parabolic {PDE}s},
   JOURNAL = {ESAIM Math. Model. Numer. Anal.},
  FJOURNAL = {ESAIM. Mathematical Modelling and Numerical Analysis},
    VOLUME = {55},
      YEAR = {2021},
    NUMBER = {1},
     PAGES = {283--299},
      ISSN = {2822-7840,2804-7214},
   MRCLASS = {65M60 (35K20 65M12 65M15)},
  MRNUMBER = {4216839},
MRREVIEWER = {Lu\ Zhang},
       DOI = {10.1051/m2an/2020084},
       URL = {https://doi.org/10.1051/m2an/2020084},
}

@article {GantnerStevenson23,
    AUTHOR = {Gantner, Gregor and Stevenson, Rob},
     TITLE = {Improved rates for a space--time {FOSLS} of parabolic {PDE}s},
   JOURNAL = {Numer. Math.},
  FJOURNAL = {Numerische Mathematik},
      YEAR = {2023},
       DOI = {10.1007/s00211-023-01387-3},
       URL = {https://doi.org/10.1007/s00211-023-01387-3},
}

@article {GantnerStevenson24,
    AUTHOR = {Gantner, Gregor and Stevenson, Rob},
     TITLE = {Applications of a space-time {FOSLS} formulation for parabolic
              {PDE}s},
   JOURNAL = {IMA J. Numer. Anal.},
  FJOURNAL = {IMA Journal of Numerical Analysis},
    VOLUME = {44},
      YEAR = {2024},
    NUMBER = {1},
     PAGES = {58--82},
      ISSN = {0272-4979,1464-3642},
   MRCLASS = {65N35 (65N12)},
  MRNUMBER = {4699574},
       DOI = {10.1093/imanum/drad012},
       URL = {https://doi.org/10.1093/imanum/drad012},
}

@article {75.64,
    AUTHOR = {Graham, I. G. and Hackbusch, W. and Sauter, S. A.},
     TITLE = {Finite elements on degenerate meshes: inverse-type
              inequalities and applications},
   JOURNAL = {IMA J. Numer. Anal.},
  FJOURNAL = {IMA Journal of Numerical Analysis},
    VOLUME = {25},
      YEAR = {2005},
    NUMBER = {2},
     PAGES = {379--407},
      ISSN = {0272-4979},
   MRCLASS = {65N38 (65N30)},
  MRNUMBER = {2126208},
MRREVIEWER = {Nasser H. Sweilam},
       DOI = {10.1093/imanum/drh017},
       URL = {http://dx.doi.org/10.1093/imanum/drh017},
}

@article {LangerSchafelner22,
    AUTHOR = {Langer, Ulrich and Schafelner, Andreas},
     TITLE = {Adaptive space-time finite element methods for parabolic
              optimal control problems},
   JOURNAL = {J. Numer. Math.},
  FJOURNAL = {Journal of Numerical Mathematics},
    VOLUME = {30},
      YEAR = {2022},
    NUMBER = {4},
     PAGES = {247--266},
      ISSN = {1570-2820,1569-3953},
   MRCLASS = {65M60 (35K20 49J20 49M41 65M15 65M50 65Y05)},
  MRNUMBER = {4519821},
MRREVIEWER = {Hamdullah\ Y\"ucel},
       DOI = {10.1515/jnma-2021-0059},
       URL = {https://doi.org/10.1515/jnma-2021-0059},
}

@article {LionsPeetre64,
    AUTHOR = {Lions, J.-L. and Peetre, J.},
     TITLE = {Sur une classe d'espaces d'interpolation},
   JOURNAL = {Inst. Hautes \'{E}tudes Sci. Publ. Math.},
  FJOURNAL = {Institut des Hautes \'{E}tudes Scientifiques. Publications
              Math\'{e}matiques},
    NUMBER = {19},
      YEAR = {1964},
     PAGES = {5--68},
      ISSN = {0073-8301,1618-1913},
   MRCLASS = {46.38 (46.10)},
  MRNUMBER = {165343},
MRREVIEWER = {M.\ Schechter},
       URL = {http://www.numdam.org/item?id=PMIHES_1964__19__5_0},
}

@book {LionsMagenes72,
    AUTHOR = {Lions, J.-L. and Magenes, E.},
     TITLE = {Non-homogeneous boundary value problems and applications.
              {V}ol. {I}},
    SERIES = {Die Grundlehren der mathematischen Wissenschaften},
    VOLUME = {Band 181},
      NOTE = {Translated from the French by P. Kenneth},
 PUBLISHER = {Springer-Verlag, New York-Heidelberg},
      YEAR = {1972},
     PAGES = {xvi+357},
   MRCLASS = {35JXX (35KXX 35LXX 46E35)},
  MRNUMBER = {350177},
}

@article{MonsuurStevensonStorn23,
    AUTHOR = {Monsuur, Harald and Stevenson, Rob and Storn, Johannes},
     TITLE = {Minimal residual methods in negative or fractional {S}obolev
              norms},
   JOURNAL = {Math. Comp.},
  FJOURNAL = {Mathematics of Computation},
    VOLUME = {93},
      YEAR = {2024},
    NUMBER = {347},
     PAGES = {1027--1052},
      ISSN = {0025-5718,1088-6842},
   MRCLASS = {65N30 (35B35 35B45)},
  MRNUMBER = {4708036},
MRREVIEWER = {Bruno\ Carpentieri},
       DOI = {10.1090/mcom/3904},
       URL = {https://doi.org/10.1090/mcom/3904},
}

@article {NeumullerSmears19,
    AUTHOR = {Neum\"uller, Martin and Smears, Iain},
     TITLE = {Time-parallel iterative solvers for parabolic evolution
              equations},
   JOURNAL = {SIAM J. Sci. Comput.},
  FJOURNAL = {SIAM Journal on Scientific Computing},
    VOLUME = {41},
      YEAR = {2019},
    NUMBER = {1},
     PAGES = {C28--C51},
      ISSN = {1064-8275,1095-7197},
   MRCLASS = {65M22 (65F10 65Y05)},
  MRNUMBER = {3904435},
MRREVIEWER = {Temur\ Jangveladze},
       DOI = {10.1137/18M1172466},
       URL = {https://doi.org/10.1137/18M1172466},
}

@book {Oswald94,
    AUTHOR = {Oswald, Peter},
     TITLE = {Multilevel finite element approximation},
    SERIES = {Teubner Skripten zur Numerik. [Teubner Scripts on Numerical
              Mathematics]},
      NOTE = {Theory and applications},
 PUBLISHER = {B. G. Teubner, Stuttgart},
      YEAR = {1994},
     PAGES = {160},
      ISBN = {3-519-02719-4},
   MRCLASS = {65N30 (65N55)},
  MRNUMBER = {1312165},
MRREVIEWER = {Weimin\ Han},
       DOI = {10.1007/978-3-322-91215-2},
       URL = {https://doi.org/10.1007/978-3-322-91215-2},
}

@article {SchwabStevenson09,
    AUTHOR = {Schwab, Christoph and Stevenson, Rob},
     TITLE = {Space-time adaptive wavelet methods for parabolic evolution
              problems},
   JOURNAL = {Math. Comp.},
  FJOURNAL = {Mathematics of Computation},
    VOLUME = {78},
      YEAR = {2009},
    NUMBER = {267},
     PAGES = {1293--1318},
      ISSN = {0025-5718,1088-6842},
   MRCLASS = {65M70 (35B30 35K90 65T60)},
  MRNUMBER = {2501051},
MRREVIEWER = {Jana\ Kopfova},
       DOI = {10.1090/S0025-5718-08-02205-9},
       URL = {https://doi.org/10.1090/S0025-5718-08-02205-9},
}

@article {SchwabStevenson17,
    AUTHOR = {Schwab, Christoph and Stevenson, Rob},
     TITLE = {Fractional space-time variational formulations of ({N}avier-)
              {S}tokes equations},
   JOURNAL = {SIAM J. Math. Anal.},
  FJOURNAL = {SIAM Journal on Mathematical Analysis},
    VOLUME = {49},
      YEAR = {2017},
    NUMBER = {4},
     PAGES = {2442--2467},
      ISSN = {0036-1410,1095-7154},
   MRCLASS = {35Q30 (35A15 65M12 65T60 76D05)},
  MRNUMBER = {3668596},
MRREVIEWER = {J\"{u}rgen\ Socolowsky},
       DOI = {10.1137/15M1051725},
       URL = {https://doi.org/10.1137/15M1051725},
}

@article {ScottZhang90,
  author        = {Scott, L.~R. and Zhang, S.},
  title         = {Finite Element Interpolation of nonsmooth Functions satisfying boundary conditions},
  fjournal       = {Mathematics of Computation},
  journal		= {Math. Comput.},  
  year          = {1990},
  volume        = {54},
  number        = {190},
  pages         = {483--493},
  month         = {4},
  DOI = {10.2307/2008497},
}

@article {SteinbachZank20,
    AUTHOR = {Steinbach, Olaf and Zank, Marco},
     TITLE = {Coercive space-time finite element methods for initial
              boundary value problems},
   JOURNAL = {Electron. Trans. Numer. Anal.},
  FJOURNAL = {Electronic Transactions on Numerical Analysis},
    VOLUME = {52},
      YEAR = {2020},
     PAGES = {154--194},
      ISSN = {1068-9613},
   MRCLASS = {65M60},
  MRNUMBER = {4102914},
MRREVIEWER = {Hidekazu\ Yoshioka},
       DOI = {10.1553/etna\{_}vol52s154}

@article {StevensonStorn22,
    AUTHOR = {Stevenson, Rob and Storn, Johannes},
     TITLE = {Interpolation operators for parabolic problems},
   JOURNAL = {Numer. Math.},
  FJOURNAL = {Numerische Mathematik},
    VOLUME = {155},
      YEAR = {2023},
    NUMBER = {1-2},
     PAGES = {211--238},
      ISSN = {0029-599X,0945-3245},
   MRCLASS = {65D05 (65M12 65M15 65M60)},
  MRNUMBER = {4646960},
       DOI = {10.1007/s00211-023-01373-9},
       URL = {https://doi.org/10.1007/s00211-023-01373-9},
}

@article {StevensonWesterdiep21,
    AUTHOR = {Stevenson, Rob and Westerdiep, Jan},
     TITLE = {Stability of {G}alerkin discretizations of a mixed space-time
              variational formulation of parabolic evolution equations},
   JOURNAL = {IMA J. Numer. Anal.},
  FJOURNAL = {IMA Journal of Numerical Analysis},
    VOLUME = {41},
      YEAR = {2021},
    NUMBER = {1},
     PAGES = {28--47},
      ISSN = {0272-4979,1464-3642},
   MRCLASS = {65M60 (35Kxx 65M12)},
  MRNUMBER = {4205051},
MRREVIEWER = {Zhicheng\ Hu},
       DOI = {10.1093/imanum/drz069},
       URL = {https://doi.org/10.1093/imanum/drz069},
}

@article{TantardiniVeeser16,
  author   = {Tantardini, F. and Veeser, A.},
  title    = {The {$L^2$}-projection and quasi-optimality of {G}alerkin methods for parabolic equations},
  journal  = {SIAM J. Numer. Anal.},
  year     = {2016},
  volume   = {54},
  number   = {1},
  pages    = {317--340},
  issn     = {0036-1429},
  doi      = {10.1137/140996811},
  fjournal = {SIAM Journal on Numerical Analysis},
  url      = {https://doi.org/10.1137/140996811},
}

@article {Tomarelli83,
    AUTHOR = {Tomarelli, Franco},
     TITLE = {Weak solution for an abstract {C}auchy problem of parabolic
              type},
   JOURNAL = {Ann. Mat. Pura Appl. (4)},
  FJOURNAL = {Annali di Matematica Pura ed Applicata. Serie Quarta},
    VOLUME = {133},
      YEAR = {1983},
     PAGES = {93--123},
      ISSN = {0003-4622},
   MRCLASS = {34G10 (35K22 46M35)},
  MRNUMBER = {725021},
MRREVIEWER = {L.\ R.\ Bragg},
       DOI = {10.1007/BF01766013},
       URL = {https://doi.org/10.1007/BF01766013},
}

@incollection {VenetieWesterdiep21,
    AUTHOR = {van Veneti\"e, Raymond and Westerdiep, Jan},
     TITLE = {A parallel algorithm for solving linear parabolic evolution
              equations},
 BOOKTITLE = {Parallel-in-time integration methods},
    SERIES = {Springer Proc. Math. Stat.},
    VOLUME = {356},
     PAGES = {33--50},
 PUBLISHER = {Springer, Cham},
      YEAR = {2021},
      ISBN = {978-3-030-75932-2},
   MRCLASS = {65M60 (65F08 65Y05 65Y20)},
  MRNUMBER = {4376430},
MRREVIEWER = {Wensheng\ Zhang},
}

@book {Wloka82,
    AUTHOR = {Wloka, Joseph},
     TITLE = {Partielle {D}ifferentialgleichungen},
    SERIES = {Mathematische Leitf\"{a}den. [Mathematical Textbooks]},
      NOTE = {Sobolevr\"{a}ume und Randwertaufgaben. [Sobolev spaces and
              boundary value problems]},
 PUBLISHER = {B. G. Teubner, Stuttgart},
      YEAR = {1982},
     PAGES = {500},
      ISBN = {3-519-02225-7},
   MRCLASS = {35-01 (46-01 65-01)},
  MRNUMBER = {652934},
MRREVIEWER = {M.\ Schechter},
}

@article {WuZheng17,
    AUTHOR = {Wu, J. and Zheng, H.},
     TITLE = {Uniform convergence of multigrid methods for adaptive meshes},
   JOURNAL = {Appl. Numer. Math.},
  FJOURNAL = {Applied Numerical Mathematics. An IMACS Journal},
    VOLUME = {113},
      YEAR = {2017},
     PAGES = {109--123},
      ISSN = {0168-9274},
   MRCLASS = {65N55 (65N30)},
  MRNUMBER = {3588590},
MRREVIEWER = {Elena Zampieri},
       DOI = {10.1016/j.apnum.2016.11.005},
       URL = {https://doi.org/10.1016/j.apnum.2016.11.005},
}

@book{Zank20,
  title={Inf-sup stable space-time methods for time-dependent partial differential equations},
  author={Zank, Marco},
  year={2020},
  publisher={Verlag der Technischen Universit{\"a}t Graz}
}

\end{document}